\documentclass[nada,11pt,preprint]{imsart}

\RequirePackage[OT1]{fontenc}
\RequirePackage{amsthm,amsmath,amssymb,latexsym,subfigure}
\usepackage{epsfig,graphics}
\RequirePackage[numbers]{natbib}
\RequirePackage[colorlinks,citecolor=blue,urlcolor=blue]{hyperref}

\def\car{\mbox{\rm{1\hspace{-0.10 cm }I}}}

\newtheorem{prop}{Proposition}[section]
\newtheorem{teor}{Theorem}[section]

\newtheorem{remark}{Remark}[section]
\newtheorem{lema}{Lemma}[section]

\startlocaldefs
\numberwithin{equation}{section}
\theoremstyle{plain}
\endlocaldefs

\begin{document}

\begin{frontmatter}
\title{A Robbins-Monro algorithm for nonparametric estimation of  NAR process with Markov-Switching: consistency}
\runtitle{Nonparametric estimation of MS-NAR process}

\begin{aug}
\author{\fnms{Lisandro} \snm{Ferm\'in}\thanksref{t1,t2, m1}\ead[label=e1]{lisandro.fermin@uv.cl}},
\author{\fnms{Ricardo} \snm{R\'{\i}os}\thanksref{m2}\ead[label=e2]{ricardo.rios@ciens.ucv.ve}}
\and\\
\author{\fnms{Luis-Angel} \snm{Rodr\'{\i}guez}\thanksref{t1,t3,m1,m3}
\ead[label=e3]{larodri@uc.edu.ve}
}

\thankstext{t1}{Support by project Anillo ACT1112}
\thankstext{t2}{Support by project DIUV REG Nº02/2011}
\thankstext{t3}{Sabbatical support from Universidad de Carabobo}
\runauthor{Ferm\'in, R\'{\i}os and Rodr\'{\i}guez}

\affiliation{Universidad de Valpara\'{\i}so\thanksmark{m1}, Universidad Central de Venezuela\thanksmark{m2}, and Universidad de Carabobo\thanksmark{m3}}

\address{
CIMFAV, Facultad de Ingenier\'{\i}a,\\
Universidad de Valpara\'{\i}so, Chile.\\
\printead{e1}\\
}

\address{
Escuela de Matem\'aticas,\\ 
Facultad de Ciencias\\ 
Universidad Central de Venezuela.\\
\printead{e2}\\
}

\address{
Dpto. de Matem\'aticas, FACYT,\\ 
Universidad de Carabobo, Venezuela.\\
\printead{e3}\\
}
\end{aug}

\begin{abstract}
We consider nonparametric estimation for autoregressive processes with Mar\-kov switching.
In this context, the Nadaraya-Watson type estimator of regression funtions is interpreted as solution 
of a local weighted least-square problem, which does not closed-form solution in the case of hidden Markov 
switching. We introduce a nonparametric recursive algorithm to approximate the estimator. Our algorithm  restores 
the missing data by means of a Monte-Carlo step and estimate the regression function via a Robbins-Monro step. 
Consistency of the estimator is proved using the strong $\alpha$-mixing property of the model. Finally, we present 
some simulations illustrating the performances of our nonparametric estimation procedure.
\end{abstract}



\end{frontmatter}

\section{Introduction}
Switching autoregressive processes with Markov re\-gi\-me can be looked at as a combination of hidden Markov models (HMM) and threshold
regression models. The switching autoregressive process have been introduced in an econometric
context by Goldfeld and Quandt (1973) \cite{Goldfeld} and they have
become quite popular in the literature ever since Hamilton (1989)
\cite{Hamilton} employed them in the analysis of the gross
internal product of the USA for two regimes: one of contraction and another of expansion. This family of models, combines different
autoregressive models to describe the evolution of the process at different periods of time, the transition between these different 
autoregressive models being controlled by a HMM. 

Switching linear autoregressive processes with Markov regime have been extensively studied and several applications in economics and 
finance can be found in, for instance, Krolzig (1997) \cite{Krolzig}, Kim and Nelson (1999) \cite{Kim-Nelson}, 
Hamilton and Raj (2003) \cite{Hamilton-Raj}. This models are also widely used in several electrical engineering areas including 
tracking of maneuvering targets, failure detection, wind power production 
and stochastic adaptive control; see for instance, Tugnait (1982), Doucet {\it et al.} (2000), Douc {\it et al.} (2005) \cite{cappe-moulines-ryden}, and Ailliot and Monbet (2012) \cite{ailliot}.

Switching non linear autoregressive models with Markov regime  have considerable interest in the statistical community, especially for
econometric series modelling. Such models were considered in particular by Francq 
and Roussignol (1997), \cite{francq-Roussignol}. They consider models that admit an additive decomposition, with particular interest in
the switching ARCH models, Franq {\it et al.} (2001) \cite{francq-Roussignol-Zakoian}. Krishnarmurthy and Ryd\'en (1998) \cite{Krish}, 
Douc {\it et al.} (2004) \cite{douc} studied an even more general 
class of switching non linear autoregressive processes that do not necessarily admit an additive decomposition. 

We consider a particular type of switching non linear autoregressive models with Markov regime, called Markov Switching Nonlinear Autorregresive process (MS-NAR) and defined by:
\begin{equation}\label{modelo}
    Y_k=r_{X_k}(Y_{k-1})+e_k
\end{equation}
where $\{e_k\}_{k\geq1}$ are i.i.d. random variables,  the sequence $\{X_k\}_{k\geq1}$
is an homogeneous Markov chain with state space $\{1,\ldots,m\}$, and $r_1(y),\ldots, r_m(y)$ are the regression functions, assumed
to be unknown.
We denote by $A$ the probability transition matrix of the Markov chain $X$, i.e. $A=[a_{ij}]$, with $a_{ij}=\mathbb{P}(X_k=j|X_{k-1}=i)$. 
We assume that the variable $Y_0$, the Markov
chain $\{X_k\}_{k\geq1}$ and the sequence $\{e_k\}_{k\geq1}$ are mutually
independents.

This model is a generalization of switching linear autoregressive models with Markov regime, also known as MS-AR model. When the regression functions $r_i$ are linear the MS-NAR process is simply a MS-AR model. 

In the parametric case, i.e. when the regression functions depend on an unknown parameter, 
the maximum likelihood estimation method is commonly used. The consistency of the maximum likelihood estimator for the MS-NAR model is given 
in Krishnarmurthy and Ryd\'en (1998) \cite{Krish}, while the consistency and asymptotic normality are proved in a more general context in Douc {\it et al.} (2004) \cite{douc}.  Several versions of the EM algorithm and its variants,
for instance SEM, MCEM, SAEM are implemented for the computations of the maximum likelihood estimator, we refer to  \cite{cappe-moulines-ryden}. 
A semiparametric estimation for the MS-NAR model was studied in R\'{\i}os 
and Rodr\'{\i}guez (2008) \cite{luis1}, where the authors consider a 
conditional least square approach for the parameter estimation and a kernel density estimator 
for the estimation of the innovation density probability.

Nevertheless, the first question that must be addressed before studying ``consistency" of any estimation procedure for this models is whether the autoregresion functions are identified. In the context of the nonparametric estimation of MS-NAR model, this problem has not yet been discussed in the literature. A treatment of this problem can be approached following the recent results given by De Castro {\it et al}. (2015) \cite{DeCastro} for nonparametric estimation of HMM models. The general idea is to identify the markov regime, and then to ensure that the estimation method provides a unique estimate. We considere nonparametric estimators obtained through the minimization of a quadratic contrast, which has a unique minimum given by the Nadaraya-Watson estimator when the Markov chain is observed, ensuring in this case that the regression functions are identified non-parametrically. However, in the case of partially observed data we can prove the identifiability up to label swapping of the hidden states. We do not address in this paper a rigorous proof of this statement, because its content is enough material for other work, moving away from the objectives and techniques outlined in this article.

In this work we consider a nonparametric regression model. That is, for $i=1,\ldots,m$, we define  a Nadaraya-Watson
type kernels estimator, given by
\begin{equation}\label{NWEstim}
\hat{r}_i(y)=\frac{\sum_{k=1}^{n-1}Y_{k+1}K\left(\frac{y-Y_k}{h}\right)\car_{i}(X_{k+1})}
{\sum_{k=1}^{n-1}K\left(\frac{y-Y_k}{h}\right)\car_{i}(X_{k+1})}.
\end{equation}
This Nadaraya-Watson type estimator was introduced for HMM models 
in  Harel and Puri \cite{Harel-Puri}.

In the first part, we establish the uniform consistency, assuming that a realization of the complete data $\{Y_k,X_k\}_{k=1,\ldots, n}$ is known; i.e. the convergence over compact subsets $\mathrm{C}\subset\mathbb{R}$, 
$$
\sup_{y\in\mathrm{C}}|\hat{r}_i(y)-r_i(y)|\to 0, \, \, \mbox{a.s.} \,\,(\mbox{when } n\to\infty).
$$

This is an interesting asymptotic result, but the key feature of MS-NAR models is that the state sequence $\{X_k\}_{k\geq1}$ is generally 
not observable, so that statistical inference has to be carried out by means of the observations $\{Y_k\}_{k\geq1}$ only. 

In the nonparametric context, the estimators of regression functions $r_{i}(y)$, for each $y$ and $i=1, \ldots, m$, can be interpreted as solutions $\theta=(\theta_1,\ldots,\theta_m)$ of the local weighted least-squares problem
$$
U(y, Y_{1:n}, X_{1:n},\theta)=\sum_{k=1}^n\sum_{i=1}^{m}K\left(\frac{y-Y_k}{h}\right)\car_i(X_{k+1})(Y_{k+1}-\theta_i)^2,
$$
where the weights are specified by the kernel $K$, so that the observations $Y_k$ near to $y$ 
has the largest influence on the estimate of the regression function at $y$. This is,
$$
\hat{\theta}(y)=\underset{\theta\in\Theta\subset \mathbb{R}^m}{\mbox{argmin}}\; U(y, Y_{1:n}, X_{1:n},\theta).
$$ 

When a realization of the state sequence $\{X_k\}_{k\geq1}$ is observed the solutions of this problem is the Nadaraya-Watson kernel estimators $\hat{r}_i$ defined in \eqref{NWEstim}. Nevertheless, when $\{X_k\}_{k\geq1}$ is a hidden Markov chain, the solution do not have closed-form, and we need to approximate it.   

In the second part, we propose a recursive algorithm for the estimation of the regression
functions $r_{i}$ with a Monte-Carlo step which restores the missing data $\{X_k\}_{k\geq1}$
by $X_{1:n}^{t}$, and a Robbins-Monro procedure in order to estimate the unknown value of $\theta$.
This approximation minimizes the potential $ U $ by the gradient algorithm
$$
\theta^{t}=\theta^{t-1}-\gamma_t\nabla_{\theta} U\left(y, Y_{1:n}^t, X_{1:n}^t,\theta^{t-1}\right),
$$
with $\{\gamma_t\}$ any sequence of real positive numbers decreasing to $0$, and $\nabla_{\theta} U$ the gradient of $U$ with respect to the vector $\theta\in \mathbb{R}^m$. 

In a general context, the Robbins-Monro approach  is study in Duflo \cite{duflo}.
Whereas EM-type algorithms with kernel estimation are used
in Benaglia {\it et. al.} \cite{benaglia} for 
finite mixtures of nonparametric multivariate densities, 
and for finite mixture of nonparametric autoregression with independents regime in Franke {\it et. al.} \cite{franke-stockis}. We 
establish the  consistency  of the estimator obtained by our Robbins-Monro algorithm. 
This asymptotic property is obtained for each fixed point $y$.  

The paper is organized as follows. In Section 2, we present
general conditions on the model that ensure the existence of a probability density distribution, the stability of the model, and we prove that it
satisfies the strong mixing dependence condition. Futhermore, we prove the uniform consistency of the Naradaya-Watson kernels estimator in the case of complete data.
Section 3, we prove the main result, the consistency  of estimator related to our Robbins-Monro algorithm. Section 4 contains some
numerical experiments on simulated data illustrating the performances of our nonparametric
estimation procedure. The some proofs are deferred to the Appendix \ref{appA}.

\section{Preliminary} 
In this section we shall review the key properties of
MS-NAR model, that we shall need for proving results. Then, we prove the uniform consistency of the Nadaraya-Watson kernel estimator when is assumed that a realization of complete data is available.

\subsection{Stability and existence of moments}

The study of the stability of the model is relatively complex for the MS-NAR model. In this section we recall known
results over the stability of this model given by Yao and Attali \cite{Yao}. Our aim is to resume the sufficient conditions 
which ensure the existence and the uniqueness of a stationary ergodic solution for the model, as well as the existence of moment 
of order $s\geq 1$ of the respective stationary distribution.  
 
\begin{enumerate}
\item[\bf E1] The Markov chain $\{X_k\}_{k\geq 1}$ is positive recurrent. Hence, it has an invariant distribution that we 
denote by $\mu=(\mu_1,\ldots,\mu_m)$.

\item[\bf E2] The functions $r_i$, for $i=1,...,m$, are continuous. 

\item[\bf E3] There are exists positive constants $\rho_i, b_i$, $i=1,...,m$,  such that for $y\in\mathbb{R}$, the following holds
$$
|r_i(y)|\leq \rho_i|y|+b_i.
$$

\item[\bf E4] $\mathbb{E}_\mu(\log(\rho(X)))=\sum_{i=1}^m\log \rho_{i}\mu_i<0$, 
where $X$ is random variable with values in $\{1,\ldots,m\}$ and distribution $\mu$.

\item[\bf E5] $\mathbb{E}(|e_1|^s)<\infty$, for some $s\geq 1$.

\item[\bf E6] The sequence $\{e_k\}_{k\geq1}$  of random variables has common density probability function $\Phi(e)$ 
with respect to the Lebesgue measure.

\item[\bf E7] There exists $b>0$ such that $\inf_{e\in\mathrm{C}}\Phi(e)>b$, where $\mathrm{C}$ is a compact set of $\mathbb{R}$.
\end{enumerate}

Condition E1 implies that $\{(Y_k,X_k)\}_{k\geq1}$
with space states $\mathbb{R}\times\{1,\ldots,m\}$ is a  Markov process. Under condition E2 this is a Feller chain and it is a strong Feller chain if in addition the condition E6 holds.

The model is called sublinear if conditions E2 and E3 hold. For MS-NAR sublinear the following result is given.

\begin{prop}[ Yao and Attali] Consider a sublinear MS-NAR. Assuming E1-E7, we have 
\begin{itemize}
\item[i)] There exists a uniqueness stationary geometric ergodic  solution.

\item[ii)] If the spectral radius of the matrix $Q_s=\left( \rho_j^s \, a_{ij}\right)_{i,j=1\ldots m}$ is strictly less than 1, with $s$ the same that in condition E4, then $\mathbb{E}(|Y_k|^s)<\infty$.
\end{itemize}
\end{prop}

\begin{remark} For the Markov chain's stability the moment condition $s\geq 1$ is enough, but for 
the asymptotic properties of kernel estimator will be necessary $s>2$.
\end{remark}

\subsection{Probability density }

We present a technical lemma where one of the results states the existence of conditional densities of model, 
and in addition we give a factorization of this density probability. This factorization will be very useful in the next sections.

Let us first introduce some notations:

\begin{itemize}
\item $V_{1:n}$ stands for the random vector $(V_1,\ldots,V_n)$, and $v_{1:n}=(v_1,\ldots,v_n)$ we mean
a realization of the respective random vector.
\item The symbol $\car_B(x)$ denotes the indicator function of set $B$, which assigns the value $1$ if $x\in B$ and  $0$ otherwise.
\item $p(V_{1:n}=v_{1:n})$ denotes the density distribution of random vector $V_{1:n}$ evaluated at $v_{1:n}$.
\end{itemize}


We consider the following assumption :

\begin{enumerate}
\item[\bf D1] The random variable $Y_0$ has a density function $p(Y_0=y_0)$ with respect to Lebesgue measure.
\end{enumerate}

The following lemma is relevant in the frame of  kernel estimation.

\begin{lema}
\label{exisdensidad}
Under conditions D1 and E6,
\begin{itemize}
\item[i)] The random vector $(Y_{0:n},X_{1:n})$ admits the probability density function
$p(Y_{0:n}=y_{0:n},X_{1:n}=x_{1:n})$, equal to
$$
\Phi(y_n-r_{x_n}(y_{n-1}))\cdots\Phi(y_1-r_{x_1}(y_{0}))a_{x_{n-1}x_n}\cdots
a_{x_{1}x_2}\mu_{x_1}p(Y_0=y_0)
$$
with respect to the product measure $\lambda\otimes\mu_c$, where $\lambda$ and $\mu_c$
denote Lebesgue and counting measures, respectively. 
\item[ii)] If $\Phi$ is a bounded density, then the joint density of $(Y_k,Y_{k'})$
satisfies 
$$
p(Y_k=y_k,Y_{k'}=y_{k'})\leq \|\Phi\|_{\infty}^2\ \mbox{for $k,k'\geq 1$.}
$$
\end{itemize}
\end{lema}

For the proof of this lemma we refered to the reader to Appendix \ref{appA}.

\subsection{Strong mixing}

A strictly stationary stochastic process $Y =\{Y_k\}_{k\in\mathbb{Z}}$ is called strongly mixing, if 
\begin{equation}
\alpha_n :=\sup\{|\mathbb{P}(A\cap B)- \mathbb{P}(A)\mathbb{P}(B)|: A\in \mathcal{M}_{-\infty}^0, B\in \mathcal{M}_n^{\infty}\} \to 0,
\end{equation}
as $n\to \infty$, where $\mathcal{M}_a^{b}$, with $a,b \in \overline{\mathbb{Z}} $, is the $\sigma$-algebra generated by $\{Y_k\}_{k= a:b}$, and is absolutely regular mixing, if 
\begin{equation}
\beta_n :=\mathbb{E}\left( ess \sup\{\mathbb{P}(B| \mathcal{M}_{-\infty}^0)-\mathbb{P}(B) :  B\in \mathcal{M}_n^{\infty}\} \right) \to 0,
\end{equation}
as $n\to \infty$.

The values $\alpha_n$ are called strong mixing coefficients, and the values $\beta_n$ are the regular mixing coefficients. For properties and examples under mixing assumptions see Doukhan \cite{dmixing}. 
In general, we have $2\alpha_n \leq \beta_n\leq 1$.

Note that the $\alpha$-mixing coefficients can be rewritten as:
\begin{equation}
\alpha_n :=\sup\{|cov(\phi,\xi)|: 0\leq \phi, \xi \leq 1, \phi \in \mathcal{M}_{-\infty}^0, \xi\in \mathcal{M}_n^{\infty}\}. 
\end{equation} 
In the case of a strictly stationary Markov process $X$, with space state $(E,\mathcal{B})$, kernel probability transition $A$ 
and invariant probability measure $\mu$, the $\beta$-mixing coefficients take the following form (see Doukhan \cite{dmixing}, section 2.4):
\begin{equation}
\beta_n :=\mathbb{E}\left( \sup\{|A^{(n)}(X,B)-\mu(B)|: B\in \mathcal{B}\} \right).
\end{equation}

\begin{prop}
\label{ModeloesMixing}
The MS-NAR model under condition E1 is strictly stationary, is $\alpha$-mixing and their coefficients $\alpha_n(Y)$ decrease 
geometrically.
\end{prop}

The proof is postponed to Appendix \ref{appA}.

\subsection{Kernel estimator: fully observed data case}

In this section we assume that a realization of the complete data  $\{Y_k, X_k\}_{k=1:n}$ 
is available. We focus in the uniform convergence over compact sets of the Nadaraya-Watson kernel estimator defined in \eqref{NWEstim}.

For a stationary MS-NAR model, the quantity of interest in the autoregression function estimation is 
$
r(y)=\mathbb{E}(Y_1|Y_0=y),
$
which can be rewritten as
$$
r(y)=\sum_{i=1}^m \mathbb{E}(Y_1|Y_0=y,X_1=i)\mathbb{P}(X_1=i).
$$
Hence, it is sufficient to estimate each autoregression function 
\begin{equation}\label{regfunction}
r_i(y)=\mathbb{E}(Y_1|Y_0=y,X_1=i), 
\end{equation}
for $i=1,\ldots,m$ and $y\in\mathbb{R}$.

Let us denote
\begin{eqnarray}
g_i(y)&:=&r_i(y)f_i(y),\\
f_i(y)&:=&\mu_ip(Y_0=y).
\end{eqnarray}
The Nadaraya-Watson kernel estimator of $r_i$ is  
$$\hat{r}_i(y)=\left\{\begin{array}{cc}
{\hat{g}_i(y)}/{\hat{f}_i(y)} & \mbox{if $\hat{f}_i(y)\not=0$,}\\
0& \mbox{otherwise}
\end{array}\right.
$$ 
with
\begin{eqnarray}
\hat{g}_i(y)&:=&\frac{1}{nh}\sum_{k=0}^{n-1}Y_{k+1}K_h\left({y-Y_k}\right)\car_{i}(X_{k+1}),\\
\hat{f}_i(y)&:=&\frac{1}{nh}\sum_{k=0}^{n-1}K_h\left({y-Y_k}\right)\car_{i}(X_{k+1}),
\end{eqnarray}
and $K_h(y)= K(y/h)$.

In order to obtain the convergence of the ratio estimator $\hat{r}_i=\hat{g}_i(y)/\hat{f}_i(y)$ we apply the method 
used by G. Collomb, see Ferraty {\it et. al} \cite{Ferraty}, which studies simultaneously the convergence of $\hat{g}_i(y)$ and $\hat{f}_i(y)$,
when $n$ tend to $\infty$. But before this, we relate the conditions that will allow us to obtain the asymptotic results.

Let us take a kernel $K:\mathbb{R}\to\mathbb{R}$, positive, symmetric, with compact support such that $\int K(t)dt=1$. 
We assume that the kernel $K$ as well as  the density $\Phi$ are bounded, i.e  
\begin{enumerate}
\item[\bf B1] $\|K\|_{\infty}<\infty$.
\item[\bf B2] $\|\Phi\|_{\infty}<\infty$.
\end{enumerate}
Under condition B1, the kernel $K$ is of order 2, i.e. $\int tK(t)dt=0$ and $0<\left|\int t^{2}K(t)dt\right|<\infty$. 

Let $\mathrm{C}$ be a compact subset of $\mathbb{R}$, we assume the following regularity conditions: 
 
\begin{enumerate}
\item[\bf R1] There exist finite constants $c, \beta>0$, such that
$$
\forall y,y'\in \mathrm{C},\ |K(y)-K(y')|
<c|y-y'|^\beta.
$$
\item[\bf R2] The density $p_0$ of $Y_0$, $\Phi$, and $r_i$ has continuous second derivatives on the interior of $\mathrm{C}$.

\item[\bf R3] For all $k\in\mathbb{N}$, the functions
    $$
    r_{i,k}(t,s)=\mathbb{E}(|Y_1Y_{k+1}|\, |Y_0=t,Y_k=s,X_1=i,X_{k+1}=i)
    $$
    are continuous. 
\end{enumerate}

We define $g_{2,i}(y):=f_i(y)r_{i,0}(y,y)= f_i(y) \mathbb{E}(Y_1^2|Y_0=y,X_1=i)$, which is continous from condition R3.

Let  $\{h_n\}_{n\geq 1}$ be a sequence of real number satisfying the following condition  
\begin{enumerate}
\item[\bf S1] For all $n\geq0$, $h_n>0$, $\lim_{n\to\infty}h_n=0$ and $\lim_{n\to\infty}nh_n=\infty$. 
\end{enumerate}

Finally, we impose one of the two following moment conditions:
\begin{enumerate}     
\item[\bf M1] $\mathbb{E}(\exp(|Y_0|))<\infty$ and $\mathbb{E}(\exp(|e_1|))<\infty$.
\item[\bf M2] $\mathbb{E}(|Y_0|^s)<\infty$ and $\mathbb{E}(|e_1|^s)<\infty$,  for some $s>2$.
\end{enumerate}

\begin{remark} Note that M1 implies M2, and M2 implies E5, which is a sufficient condition for the stability of MS-NAR model. 

From independence between $Y_0$ and $e_1$, condition M1 implies 
$$
\mathbb{E}(\exp(|Y_1|))\leq c\mathbb{E}(\exp(|Y_0|))\mathbb{E}(\exp(|e_1|)).
$$
Moreover, E3 and M2 imply $\mathbb{E}(|Y_1|^s)< \infty$, this condition is also implied by M1.

Condition M1 is assumed in order to obtain the a.s. uniform convergence over compact sets and M2 for the a.s pointwise convergence.
\end{remark}

Now, we establish the uniform convergence over compact sets of the Nadaraya-Watson kernel estimator $\hat{r}_i$ defined in \eqref{NWEstim}. For this, we need the following three technical lemmas. Their  respectives  proofs are reported to Appendix \ref{appA}. 

The first lemma allows to treat in a unified way the asymptotic behavior of the variances and covariances of $\hat{f}_i$ and  
a truncated version of $\hat{g}_i$. The others two lemmas given the asymptotic bound for the bias and variance term in the estimation of 
the regression functions $r_i$'s. 

\begin{lema}\label{asympt_cov} Assume the model MS-NAR satisfying conditions E1-E2, E5-E6, D1, B1-B2, R3 and S1. Let 
$$
T_{k,n}=aK_h(y-Y_k)\car_i(X_{k+1})+bY_{k+1}\car_{\{|Y_{k+1}|\leq M_n\}}K_h(y-Y_k)\car_i(X_{k+1}).
$$
Then, the following statements hold:
\begin{itemize}
\item[i)]  $\mathrm{var}(T_{0,n})\approx h(a^2f_i(y)+2abg_i(y)+b^2g_{2,i}(y))\int K^2(z)dz + o(h^2)$.
\item[ii)] $\mathrm{cov}(T_{0,n},T_{k,n})\preceq h^2(a^2c_1+2abc_2(y)+b^2c_3(y))+o(h^3)$.
\end{itemize}
For all $n>0$,
\begin{itemize}
\item[iii)] $\mathrm{cov}(T_{0,n},T_{k,n})\leq (a^2 + 2ab M_n+b^2M_n^2)4\|K\|_{\infty}^2 \alpha_{k-1}$.
\end{itemize}
\end{lema}

\begin{lema} 
\label{cvarianza}
Assume that the model MS-NAR satisfies conditions E1-E2, E5-E6, D1, B1-B2, R2-R3, S1 and M2. 
Let $\{M_n\}_{n\geq 1}$ be a positive sequence and $\delta > 1$, then the following asymptotic inequalities hold true.
\begin{itemize}
\item[i)]
$\mathbb{P}(|\hat{g}_i(y)-\mathbb{E}\hat{g}_i(y)|\geq \epsilon)\preceq 4\left(1+\frac{\epsilon^2nh}{16 \delta}\right)^{-\delta/2}
+c_1\frac{16M_n\zeta^{u_n}}{\epsilon h} +  
c_2 \frac{M_n^{(2-s)}}{\epsilon^2 h^2}
$,
\item[ii)] 
$\mathbb{P}(|\hat{f}_i(y)-\mathbb{E}\hat{f}_i(y)|\geq \epsilon)\preceq 4\left(1+\frac{\epsilon^2nh}{16 \delta}\right)^{-\delta/2}
+c_3\frac{16\zeta^{u_n}}{\epsilon h}$.
\end{itemize}
\end{lema}

\begin{lema}
\label{lsesgo}
Assume that the model MS-NAR satisfies conditions E1, E6, D1, B1 and R2. 
Then the following statements hold true.
\begin{itemize}
\item[i)] $\sup_{y\in\mathrm{C}}|\mathbb{E}\hat{g}_i(y)-g_i(y)|=O(h^2)$.
\item[ii)] $\sup_{y\in\mathrm{C}}|\mathbb{E}\hat{f}_i(y)-f_i(y)|=O(h^2).$
\end{itemize}
\end{lema}
  
\begin{remark}
Lemma  \ref{asympt_cov} is a preliminary result  necessary in order to prove the Lemma \ref{cvarianza}.
\end{remark}

\begin{teor}
\label{conuniforme}
Assume that the model MS-NAR \eqref{modelo} satisfies conditions E1-E4, E6-E7, D1, B1-B2, R1-R3 and S1.
%
%
Then,
\begin{enumerate}
\item[i)] If $nh/\log n \to \infty$ and condition M2 holds,
$$
|\hat{r}_i(y)-r_i(y)|\to 0\quad \mbox{a.s.}
$$
\item[ii)] If $nh/\log n \to \infty$ and condition M1 holds,
$$
\sup_{y\in\mathrm{C}}|\hat{r}_i(y)-r_i(y)|\to 0\quad \mbox{a.s.}
$$
\end{enumerate}
\end{teor}

The proof of theorem is reported to Appendix \ref{appA}. 


\section{Main results}
In this section we present our Robbins-Monro type algorithm for the nonparametric estimation of MS-NAR model in the partial observed data case, and we prove the consistency of the estimator.

The Nadaraya-Watson estimator $\hat{r}(y)=(\hat{r}_1(y),\ldots, \hat{r}_m(y) )$, for each $y$, can be interpreted as the
solution of a local weighted least-squares problem, in our case this consists to find the minimum of the potential $U$ defined by 
\begin{equation}
\label{potencialU}
U(y,Y_{1:n},X_{1:n},\theta)=\sum_{k=1}^n\sum_{i=1}^{m}K_h(y-Y_k)\car_i(X_{k+1})(Y_{k+1}-\theta_i)^2,
\end{equation}
with respect to $\theta=(\theta_1,\ldots,\theta_m)$ in a convex open set $\Theta$ of $\mathbb{R}^{m}$. 
Thus, the regression estimator $\hat{r}$ is given by 
$$
\hat{r}(y)=\underset{\theta\in\Theta \subset \mathbb{R}^m}{\mbox{argmin}}\;\;
U(y,Y_{1:n},X_{1:n},\theta).
$$

In the partial observed data case, this is when we not observe $\{X_k\}_{k\geq 1}$, we cannot obtain an explicit expresion for the
solution $\hat{r}(y)$. Then, we must consider  
a recursive algorithm for the aproximation of this solution. Our approach approximates the estimator $\hat{r}(y)$ by a
stochastic recursive algorithm similar to that of Robbins–Monro,
\cite{cappe-moulines-ryden,duflo,YaoRe}.
This involves two steps: first a Monte-Carlo step which restores the
missing data $\{X_n\}_{n\geq 1}$, and a second step where we consider a Robbins-Monro's approximation in order to minimize the potential $U$.

Here are some further notations that we will take into account.
\begin{itemize}
\item For each $1\leq i\leq m$, $n_i=\sum_{k=1}^n\car_{i}(X_k)$ is the number of visits of the Markov chain $\{X_k\}_{\geq1}$ 
to state $i$ in the first $n$ steps, and $n_{ij}(X_{1:n})=\sum_{k=1}^{n-1}\car_{i,j}(X_{k-1},X_k)$ is the number of transitions from $i$ 
to $j$ in the first $n$ steps.
\vspace{0.2cm}
\item $\psi^{t}=(\theta^t,A^t)$ is a vector containing the estimated functions $\theta^t=(\theta_1^t,\ldots, \theta_m^t)$ and the estimated probability transition matrix $A^t$, in the t-th iteration of the Robbins-Monro algorithm.
\end{itemize}

\subsection{Restoration-estimation Robbins-Monro algorithm}
\begin{description}
 \item[Step 0.] Pick an arbitrary initial realization  $X_{1:n}^{0}=X_1^{0},\ldots,X_n^{0}$. Compute the estimated regresion functions $\hat{r}^{0}(y)=(\hat{r}_1^{0}(y),\ldots,\hat{r}_m^{0}(y))$ from equation \eqref{NWEstim} in term of the observed data $Y_{1:n}$ and the initial realization $X_{1:n}^{0}$, and compute the estimated transition matrix  $A^{0}$, by  $A_{ij}^{0}=n_{ij}(X_{1:n}^{0})/n$ for $i,j=1,\ldots,m$.
 Define $\theta^{0}=\hat{r}^{0}(y)$.
\end{description}
For $t\geq1$,
\begin{description}
\item[Step R.] Restore the corresponding unobserved 
data by drawing an sample $X_{1:n}^{t}$ from the conditional
distribution $p(X_{1:n}|Y_{0:n},\psi^{t-1})$.
\item[Step E.]  Update the estimation $\psi^{t}=(\theta^t,A^t)$ by
\begin{equation}
\label{rbnopar}
\theta^{t}=\theta^{t-1}-\gamma_t\nabla_\theta U\left(y,Y_{1:n},X_{1:n}^{t},\theta^{t-1}\right),
\end{equation}
where 
$\nabla_\theta U\left(y,Y_{1:n},X_{1:n}^{t},\theta^{t-1}\right)=\nabla_\theta U\left(y,Y_{1:n},X_{1:n}^{t},\theta\right)
\bigm|_{\theta=\theta^{t-1}}$, and $A^{t}=n_{ij}(X_{1:n}^{t})/n$.

\item[Step A.] Reduce the asymptotic variance of the algorithm by
using the averages $\bar{\theta}^{t}1/t\sum_{k=1}^t\theta^k$ instead of $\theta^{t}$, which can be
recursively computed by $\bar{\theta}^{0}=\theta^0$, and 
\begin{equation}
\label{promedio}
\bar{\theta}^{t}=\bar{\theta}^{t-1}+\frac{1}{t}\left(\theta^{t}-\bar{\theta}^{t-1}\right).
 \end{equation}
\end{description}

The following result enables us to write the algorithm 
as a stochastic gradient algorithm. Let,
$
\mathbb{E}_{\theta'}\left(U(y,Y_{1:n},X_{1:n}^{t},\theta)|\mathcal{F}_{t-1} \right)=u(y,Y_{1:n},\theta)
$
with  ${E}_{\theta'}(\cdot)={E}(\ |{\theta'})$ and $\mathcal{F}_{t-1}$ the $\sigma$-algebra generated by $\{X_{1:n}^s\}_{s=1:(t-1)}$. The proof is given in Appendix \ref{appA}. 

\begin{lema}\label{lema_GradEst}
For each $\theta\in\Theta$ we have,
\begin{equation}
\label{constrastelimite}
u(y,Y_{1:n},\theta)=\frac{1}{nh}\sum_{k=1}^n\sum_{i=1}^{m}K_h(y-Y_k)\mathbb{P}(X_{k+1}=i|Y_{0:n},\theta')(Y_{k+1}-\theta_i)^2
 \end{equation}
and
$
\mathbb{E}_{\theta'}\left(\nabla_\theta U(y,Y_{1:n},X_{1:n}^{t},\theta)|\mathcal{F}_{t-1}\right)=\nabla_\theta u(y,Y_{1:n},\theta).
$
\end{lema}

Therefore, the Restauration-Estimation  algorithm is a stochastic gradient algorithm
that minimizes $u(y,Y_{1:n},\theta)$ and it can be written as
\begin{equation}
\label{GradienteEst}
\theta^{t}=\theta^{t-1}+\gamma_t\left(-\nabla_\theta u(y,Y_{1:n},\theta^{t-1})+\varsigma_t\right),
\end{equation}
where
\begin{eqnarray*}
\varsigma_t
&=&-\nabla_\theta U\left(y,Y_{1:n},X_{1:n}^{t},\theta^{t-1}\right)+
\mathbb{E}_{\theta'}\left(\nabla_\theta U\left(y,Y_{1:n},X_{1:n}^{t},\theta^{t-1}\right)|\mathcal{F}_{t-1}\right)\\
&=&-\nabla_\theta U\left(y,Y_{1:n},X_{1:n}^{t},\theta^{t-1}\right)+\nabla_\theta u(y,Y_{1:n},\theta^{t-1}). 
\end{eqnarray*}

So, the stochastic gradient algorithm is obtained by perturbation of the following gradient system
$$
\dot{\theta}=-\nabla_\theta u(y,Y_{1:n},\theta).
$$

In the following we describe in detail each step of the algorithm.

\subsubsection*{Step 0: SAEM algorithm}
\label{paso0}

We use a Stochastic Approximation version of EM algorithm, proposed by Delyon et al. \cite{Lavielle}, 
in order to  maximize the likelihood of the data. Assume 
that the regression functions are linear and the noise is gaussian. This algorithm proved to be more
computationally efficient than a classical Monte Carlo EM algorithm due to the recycling of simulations from one iteration to the next 
in the smoothing phase of the algorithm.
The used SAEM algorithm is detailed in Section 11.1.6 of \cite{cappe-moulines-ryden}.

\subsubsection*{Step R: Carter and Kohn filter}
\label{pasoES}
The R step of the algorithm
corresponds to conditional simulation given $\psi^{t-1}$.
We describe the sampling method for the conditional distribution
\begin{eqnarray*}
\lefteqn{ p(X_{1:n}=x_{1:n}|Y_{0:n}, \psi^{t-1})}\\
\nonumber &=&\frac{\lambda_{x_1}p(Y_1|Y_0,X_1=x_1,\psi^{t-1})
    \ldots
    a_{x_{n-1}x_{n}}p(Y_n|Y_{n-1},X_n=x_n,\psi^{t-1})}{p(Y_{1:n}|Y_0,\psi^{t-1})},
\end{eqnarray*}
for all $x_{1:n}\in\{1,\ldots,m\}^N$.

Carter and Kohn \cite{Carterkohn} obtained samples $X_{1:n}$ following a stochastic version of the
hidden Markov model forward-backward algorithm first proposed by Baum {{\it et al}.} \cite{Baum70}. 
This follows by noting that 
$p(X_{1:n}|Y_{0:n},\psi^{t-1})$ can be decomposed as
    $$p(X_{1:n}|Y_{0:n},\psi^{t-1})=
    p(X_n|Y_{0:n},\psi^{t-1})
    \prod_{k=1}^{n-1}p(X_k|X_{k+1},Y_{0:n},, \psi^{t-1}).
    $$
Provided that $X_{k+1}$ is known, $p(X_k|X_{k+1},Y_{0:n}, \psi^{t-1})$ is a discrete distribution, suggesting 
the following sampling strategy: for
 $k = 2,\ldots, n$ and $i=1,\ldots,m$, compute recursively the optimal filter as 
\begin{eqnarray*}
\lefteqn{ p(X_k=i|Y_{0:k},\psi^{t-1})}\\
    &\propto& p(Y_k|Y_{k-1},X_k=i, \psi^{t-1})
    \sum_{i=1}^{m}a_{ij}p(X_{k-1}=j|Y_{1:k-1},\psi^{t-1}).
\end{eqnarray*}
    
Then, sample $X_n$ from $p(X_n|Y_{0:n},\psi^{t-1})$ and for 
$k=n-1,\ldots,1$, $X_k$ is
sampled from
    $$
    p(X_k=i|X_{k+1}=x_{k+1},Y_{0:k},\psi^{t-1})
    =\displaystyle\frac
    {a_{i{x_{k+1}}}p(X_k=i|Y_{0:k},\psi^{t-1})}
    {\sum_{l=1}^{m}a_{il}p(X_k=l|Y_{0:k},\psi^{t-1})}.
    $$
    
Following the proof in Rosales \cite{Rosales}, we will show that the sequence $\{X_{1:n}^{t}\}$ 
is an ergodic Markov chain  with 
invariant distribution $p(X_{1:n}=x_{1:n}|Y_{0:n},\psi^*)$.
It is sufficient to note that the sequence $\{X_{1:n}^{t}\}_{t\geq 1}$ is an irreducible and aperiodic Markov chain 
on a finite state space, $\{1,\ldots,m\}^N$. Irreducibility and
aperiodicity follow directly from the positivity of the kernel,
    \begin{eqnarray*}
    Q(X_{1:n}^{t}|X_{1:n}^{t-1},Y_{0:n}, \psi^t)\!\propto&\!\!p(X_n^{t}|Y_{0:n},\psi^{t-1})
    \prod_{n=1}^{N-1}
    p(X_n^{t}|X_{n+1}^{t},Y_{0:n},\psi^{t-1})>0.\
    \end{eqnarray*}

In this case the standard ergodic result for finite
Markov chains applies, (Kemeny y Snell \cite{Kemeny}) 
\begin{equation}
\label{ConverFiltro}
\|Q(X_{1:n}^{t}|X_{1:n}^{t-1},Y_{0:n},\psi^{t-1})-p(X_{1:n}|Y_{0:n},\psi^*)\|\leq c\zeta^{t}, 
\end{equation}

Moreover, \eqref{ConverFiltro} is satisfied with
$c=card(\{1,\ldots,m\}^N)$, $\zeta=(1-2Q_x)$ and 
$$Q_x=\inf_{x',\psi'}Q(x'|x,\psi'),$$
for $x,x'\in\{1,\ldots,m\}^N$.

\subsubsection*{Step E: Estimation}
In each iteration of this algorithm,
we evaluate $\nabla_\theta U(y,Y_{1:n},X_{1:n},\theta)$ the gradient of the potential. 
For each $1\leq i\leq m$, we compute the components
$$
\frac{\partial U}{\partial\theta_i}(y,Y_{1:n},X_{1:n},\theta)=\hat{g}_i(y,Y_{1:n},X_{1:n})-\theta_i\hat{f}_i(y,Y_{1:n},X_{1:n}).
$$
In each iteration this quantity is updated. It has the advantage of that the ratio $\hat{r}_i$ is not computed directly, avoiding the zeros of the function $\hat{f}_i$.

\subsubsection*{Step A: Average (or Aggregation)}
To reduce the asymptotic variance of the  estimated parameters $\{\theta^t\}$, we adopt the averaging
technique introduced by Polyak and Juditsky (see  \cite{Polyakuditsky}). 
The idea is to use the averages $\{\bar{\theta}^t\}$
defined by $\bar{\theta}^t=1/t\sum_{k=1}^t\theta^k$ instead of $\{{\theta}^t\}$, which can be
recursively computed by means of the equation \eqref{promedio}. 

\subsection{Consistency}
The convergence analysis of Robbins-Monro approximations
are well studied in Duflo \cite{duflo} in the general case. In this paper we
use a similar framework as in  Cappe {\it et. al.} (\cite{cappe-moulines-ryden},
p\'ag. 431), for the convergence of the stochastic gradient algorithm for the  likelihood function in hidden Markov models, considering that in our particular case $u(\theta)$ 
is a con\-ti\-nuous\-ly differentiable function of  $\theta$. The following  convergence result is given for each $y$ fixed.

\begin{teor}\label{conv_RM}
Assume that $\{\gamma_t\}$ is a positive sequence such that
$$\sum_t \gamma_{t}=\infty,\quad \ \sum_t\gamma_{t}^2<\infty,$$
and that the closure of the set $\{\bar{\theta}^{t}\}$ is a compact subset of
$\Theta$. Then, almost surely, the sequence  $\{\bar{\theta}^{t}\}$ 
satisfies $\lim_{t\to\infty}\nabla_{\theta}u(y,Y_{1:n},\bar{\theta}^t)=0$. Furthermore, 
$\lim_{t\to\infty}{\bar{\theta}}^t=\theta^*$ and $\nabla_{\theta}u(y,Y_{1:n},\theta^*)=0$, a.s.
\end{teor}

\begin{proof}
Let $M^t=\sum_{s=0}^t\gamma_s\varsigma_s$. 
The sequence $\{M^t\}$ is an $\mathcal{F}_t$ martingale, in fact
\begin{eqnarray*}
\mathbb{E}(M^t|\mathcal{F}_{t-1})&=&\mathbb{E}(\gamma_{t}\varsigma_t+M^{t-1}|\mathcal{F}_{t-1})\\
&=&\mathbb{E}(\gamma_{t}\varsigma_t|\mathcal{F}_{t-1})+\mathbb{E}(M^{t-1}|\mathcal{F}_{t-1})\\
&=&M^{t-1}.
\end{eqnarray*}
Moreover, it satisfies 
 $\sum_{t=1}^{\infty}\mathbb{E}(\|M^{t}-M^{t-1}\|^2|\mathcal{F}_{t-1})< \infty$.  Indeed, 
$$
\mathbb{E}(\|M^{t}-M^{t-1}\|^2|\mathcal{F}_{t-1})= \gamma_{t}^2\mathbb{E}(|\varsigma_{t}|^2|\mathcal{F}_{t-1})
$$
and
\begin{eqnarray*}
\|\varsigma_{t}\|^2 &=& \sum_{i=1}^{m}\left(\frac{\partial U}{\partial\theta_i}(y,Y_{1:n},X^{t-1},\theta^{t-1})
-\frac{\partial u}{\partial\theta_i}(y,Y_{1:n},\theta^{t-1})\right)^2\\
&=&\frac{4}{n^2h^2}\sum_{i=1}^{m}\left(\sum_{k=1}^n(Y_{k+1}-\theta_i^{t-1}) K_h(y-Y_k)\left(B_i^t(k)|\mathcal{F}_{t-1})\right)\right)^2,
\end{eqnarray*}
where $B_i^t(k)=\car_i(X_{k+1}^{t})- \mathbb{E}(\car_i(X_{k+1}^{t})$ are Bernoulli centered random variables. Then, $\mathbb{E}(\|\varsigma_{t}\|^2|\mathcal{F}_{t-1})$ is 
\begin{eqnarray*} 
\frac{4}{n^2h^2}\sum_{i=1}^{m}\sum_{k,k'=1}^n(Y_{k+1}-\theta_i^{t-1})(Y_{k'+1}
-\theta_i^{t-1})K_h(y-Y_k)K_h\left({y-Y_{k'}}\right)\chi_i^t(k,k'),
\end{eqnarray*}
with $\chi_i^t(k,k')=\mbox{cov}(\car_i(X_{k+1}^{t}),\car_i(X_{k'+1}^{t})|\mathcal{F}_{t-1})$. Thus, by Cauchy-Schwarz's inequality
we have 
$$
\chi_i^t(k,k')\leq
\sqrt{\mbox{var}(B_i^t(k)|\mathcal{F}_{t-1})}\sqrt{\mbox{var}(B_i^t(k')|\mathcal{F}_{t-1})}\leq 1/4,
$$ 
and
\begin{eqnarray}\label{moment2_marting}
\mathbb{E}(\|\varsigma_{t}\|^2|\mathcal{F}_{t-1})&\leq& \frac{1}{n^2h^2}\sum_{i=1}^{m}\left(\sum_{k=1}^n(Y_{k+1}-\theta_i^{t-1})K_h(y-Y_k)\right)^2\\ 
\nonumber &=& \| \Psi(\theta^{t-1})\|^2,
\end{eqnarray}
where $\Psi(\theta)=(\Psi_1(\theta),\ldots,\Psi_m(\theta))$ and
$$
\Psi_i(\theta)= \frac{1}{n^2h^2}\left(\sum_{k=1}^n(Y_{k+1}-\theta_i)K_h(y-Y_k)\right)^2.
$$
By compactness $\|\Psi(\theta)\|^2$ is finite, therefore
$$
\mathbb{E}(\|M^{t}-M^{t-1}\|^2|\mathcal{F}_{t-1})\leq \|\Psi(\theta^{t-1})\|^2\sum_{t=1}^{\infty}\gamma_{t}^2<\infty.
$$ 
Thus, by applying conditional Borel-Cantelli lemma in Cappe {\it et. al} (\cite{cappe-moulines-ryden},  Lemma 11.2.9)  
the sequence $\{M^t\}$ has a finite limit a.s. and according to Theorem 11.3.2 in \cite{cappe-moulines-ryden} 
the sequence $\{\theta^t\}$ satisfies 
$$
\lim_{t\to\infty}\nabla_{\theta}u(y,Y_{1:n},{\theta}^t)=0.
$$
By continuity of the function $\nabla_{\theta}u$ we proved that $\theta^*=\lim_{t\to\infty}{\theta}^t$
satisfies $\nabla_{\theta}u(y,Y_{1:n},\theta^*)=0$, and by Ces\`aro theorem, $\lim_{t\to\infty}\bar{\theta}^t=\theta^*$.
\end{proof} 

The critical point $\theta^{*}\in \Theta\subset \mathbb{R}^m$ of the gradient of $u$, has $i$-th component $\theta_i^{*}$ given by
$$
\theta_i^{*}(y,Y_{1:n})=\frac{\sum_{k=1}^{n-1}Y_{k+1}K_h(y-Y_k)\mathbb{P}(X_{k}=i|Y_{0:n},\theta)}
{\sum_{k=1}^{n-1}K_h(y-Y_k)\mathbb{P}(X_{k}=i|Y_{0:n},\theta)}.
$$
We can prove, following the Theorem \ref{conuniforme}, that $ \theta_i^{*}(y,Y_{1:n})\to r_i(y)$, {\it a.s.}, whenever
$n\to \infty$.  As 
a consequence we obtain
$$
\lim_{n\to\infty}\lim_{t\to\infty}\theta_i^{t}(y,Y_{1:n})=r_i(y),\ a.s.
$$


\section{Numerical examples}
In this section we illustrate the performances of the algorithms developed in the previous section by applying
them to simulated data. We work a MS-NAR  with $m = 2$ states and autoregressive functions
$$
r_1(y)=0.7y+2e^{(-10y^2)},\ \ r_2(y)=\frac{2}{1+e^{10y}}-1,
$$
where $r_1$ is a bump function and $r_2$ is a decreasing logistic function. These 
functions was considered by Franke {\it et. al. } \cite{franke-stockis}. 
Let $\Phi$  be a gaussian white noise with variance $\sigma^2=0.4$.
The transition probability matrix is given by 
$$
A=\left(\begin{array}{cc}
 0.98 & 0.02\\
0.02 & 0.98
\end{array}\right).
$$

We used a straightforward implementation of the algorithms described.
We generate a sample of length $n = 1000$. For each $k$, we simulate  $X_k$ and then  use it to simulate $Y_k$. The simulated data is plotted in Figure \ref{fig1} (left).

\begin{figure}[h]
\centering
\subfigure{
\includegraphics[width=0.5\textwidth]{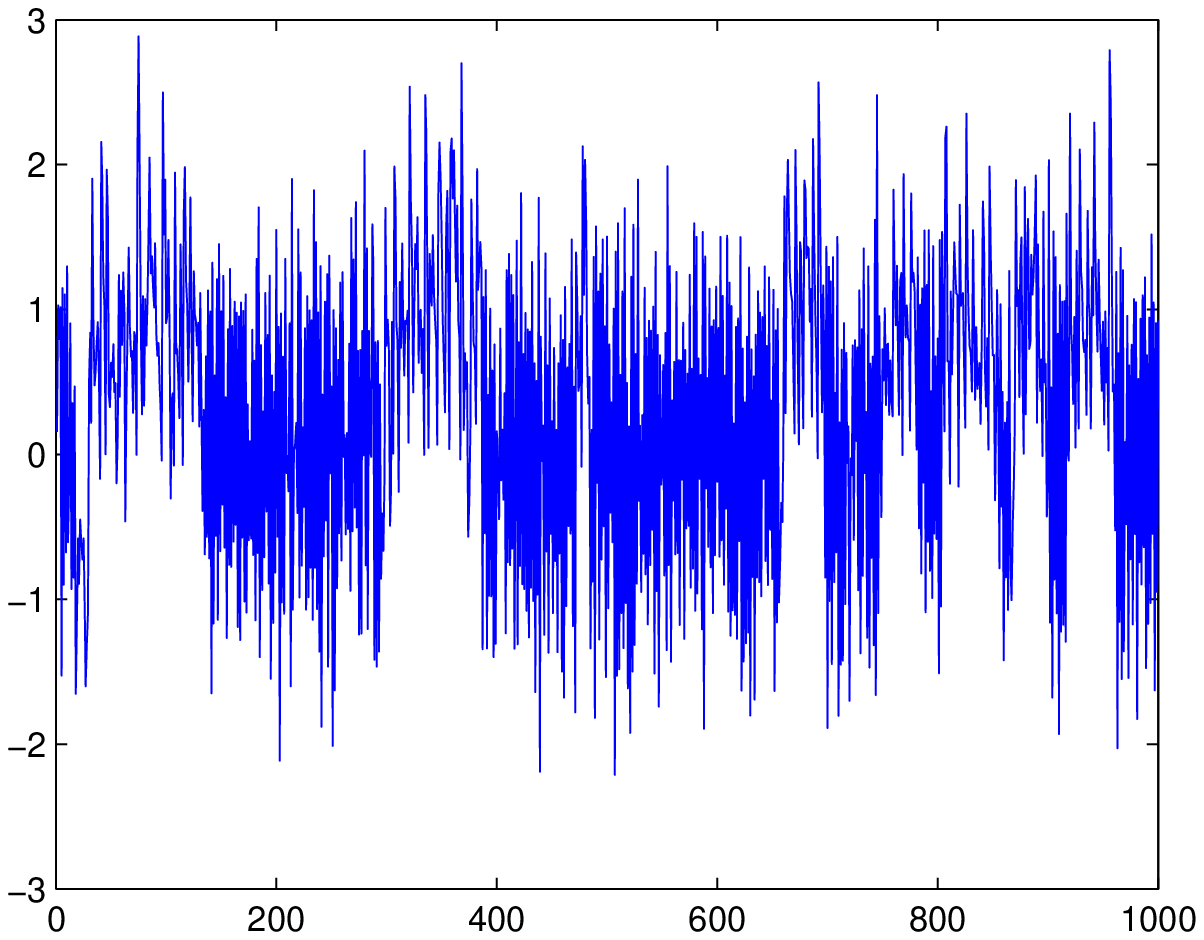}
}\hspace{-0.6cm}
\subfigure{
\includegraphics[width=0.5\textwidth]{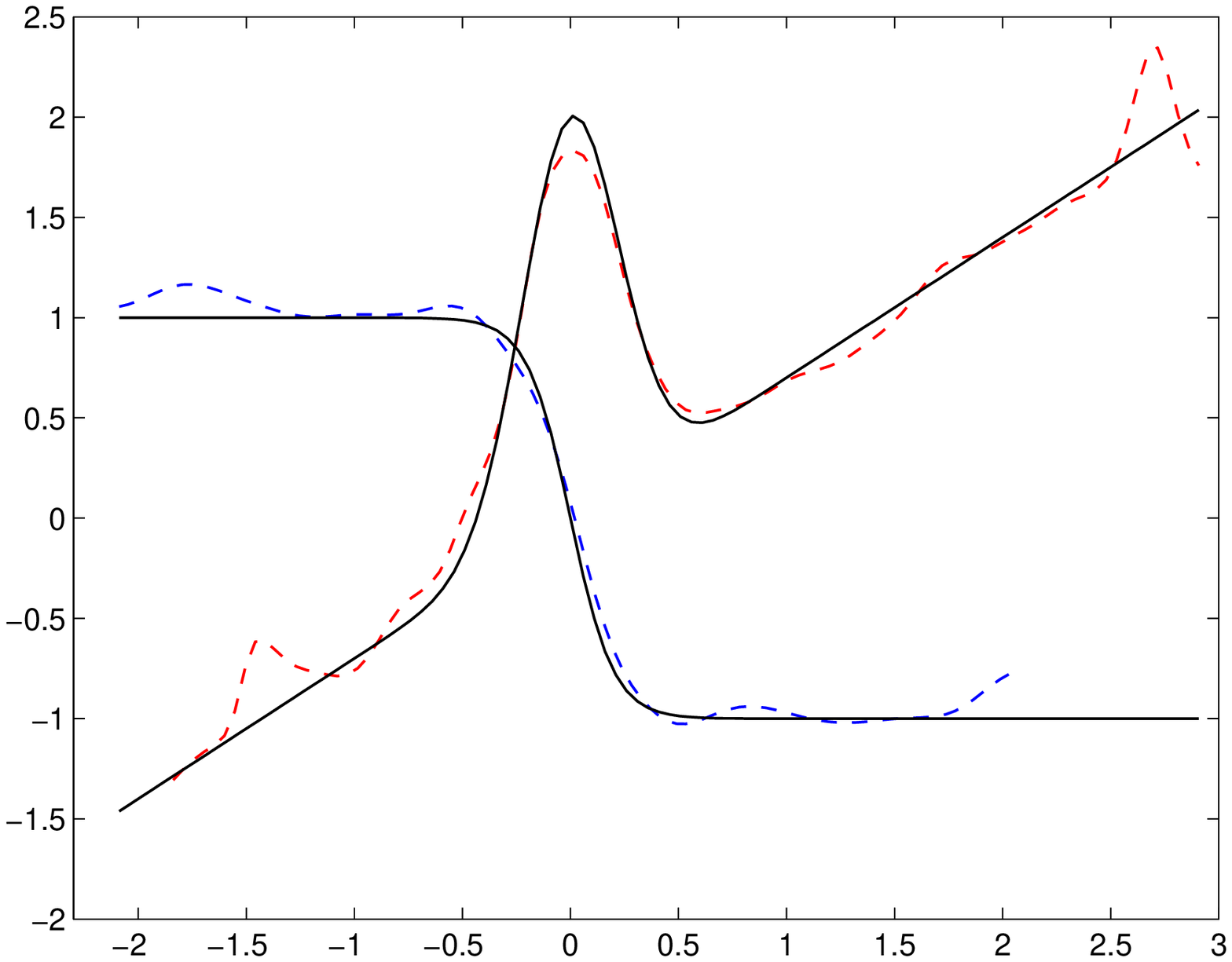}
}
\caption{Simulated data $Y_{1:n}$ (left). Estimated kernel functions for complete data $(X_{1:n}, Y_{1:n})$ (right).
}\label{fig1}
\end{figure}

For the estimation of the regression function $r_i$, we use the standard gaussian density as the kernel function $K$, in spite of the fact that it is not compactly supported. As bandwidth parameter we take $h=\left(n/\log(n)\right)^{1/5}$. 

Assuming that the complete data $\{Y_k,X_k\}_{k=1:1000}$ is available, we show in Figure 
\ref{fig1} (right) the performance of  $r_1$ and $r_2$ (solid curved line) and their respective kernel estimates (dotted curved line).

We implemented the Restauration-Estimation algorithm for the data described above.
The initial estimates for the Markov chain $X_{1:n}^0$ in the step 0 of our algorithm
was obtain by using a SAEM algorithm for the MS-AR  model,
\begin{equation*}
 Y_n=\rho_ {X_n}Y_{n-1}+b_{X_{n}}+\sigma_{X_n}e_n.
\end{equation*}
The parameters estimates obtained using this implementation are,
$$
\hat{A}=\left(\begin{array}{cc}
  0.983 & 0.017\\
0.017 & 0.983
\end{array}\right),
$$
and the linear functions estimates are of the form $\hat{r}_1(y)=-0.8239y-0.0218$ and
$\hat{r}_2(y)=0.2943y+ 0.6334$.

\begin{figure}[h]
\centering
\subfigure{
\includegraphics[width=0.5\textwidth]{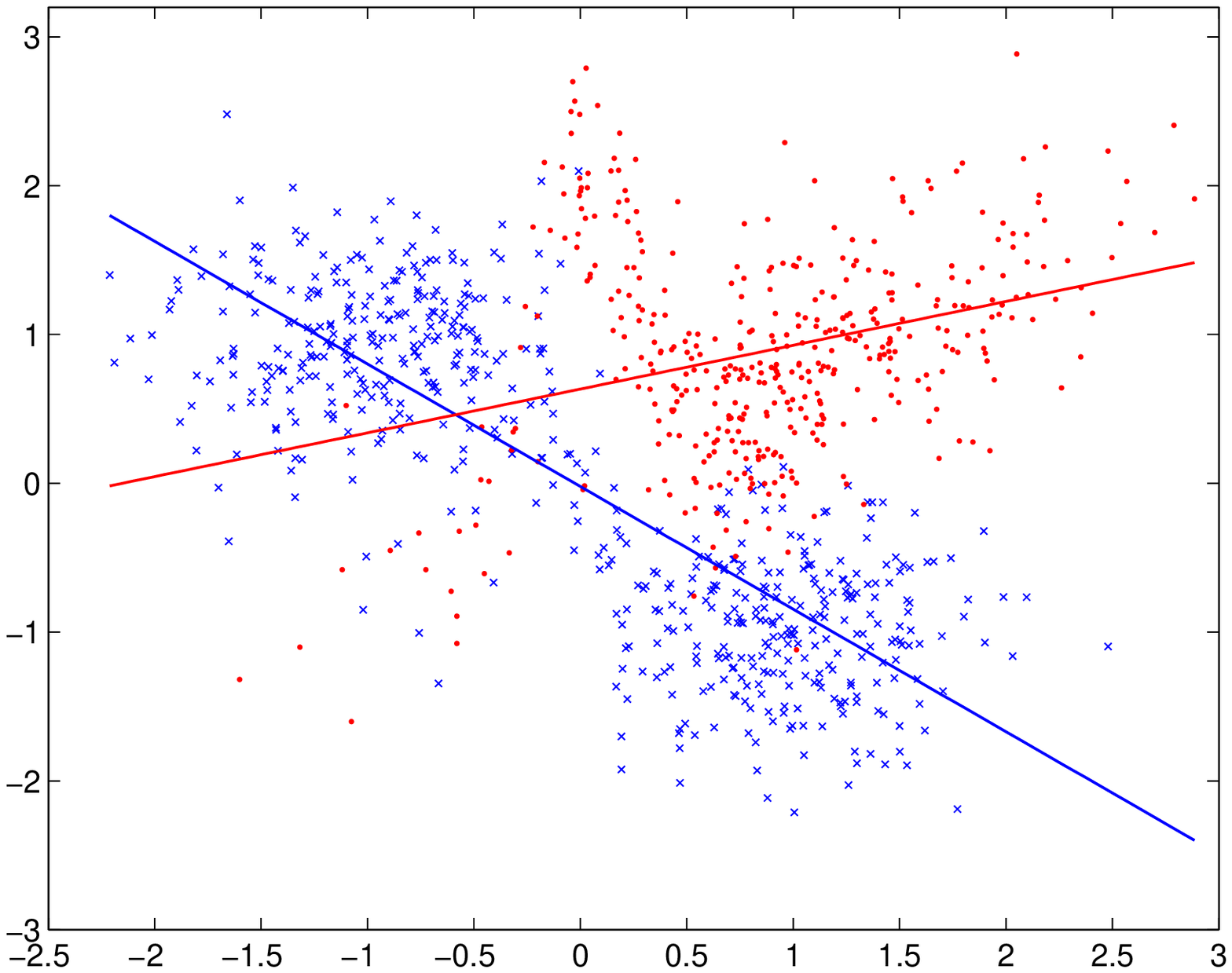}
}\hspace{-0.6cm}
\subfigure{
\includegraphics[width=0.5\textwidth]{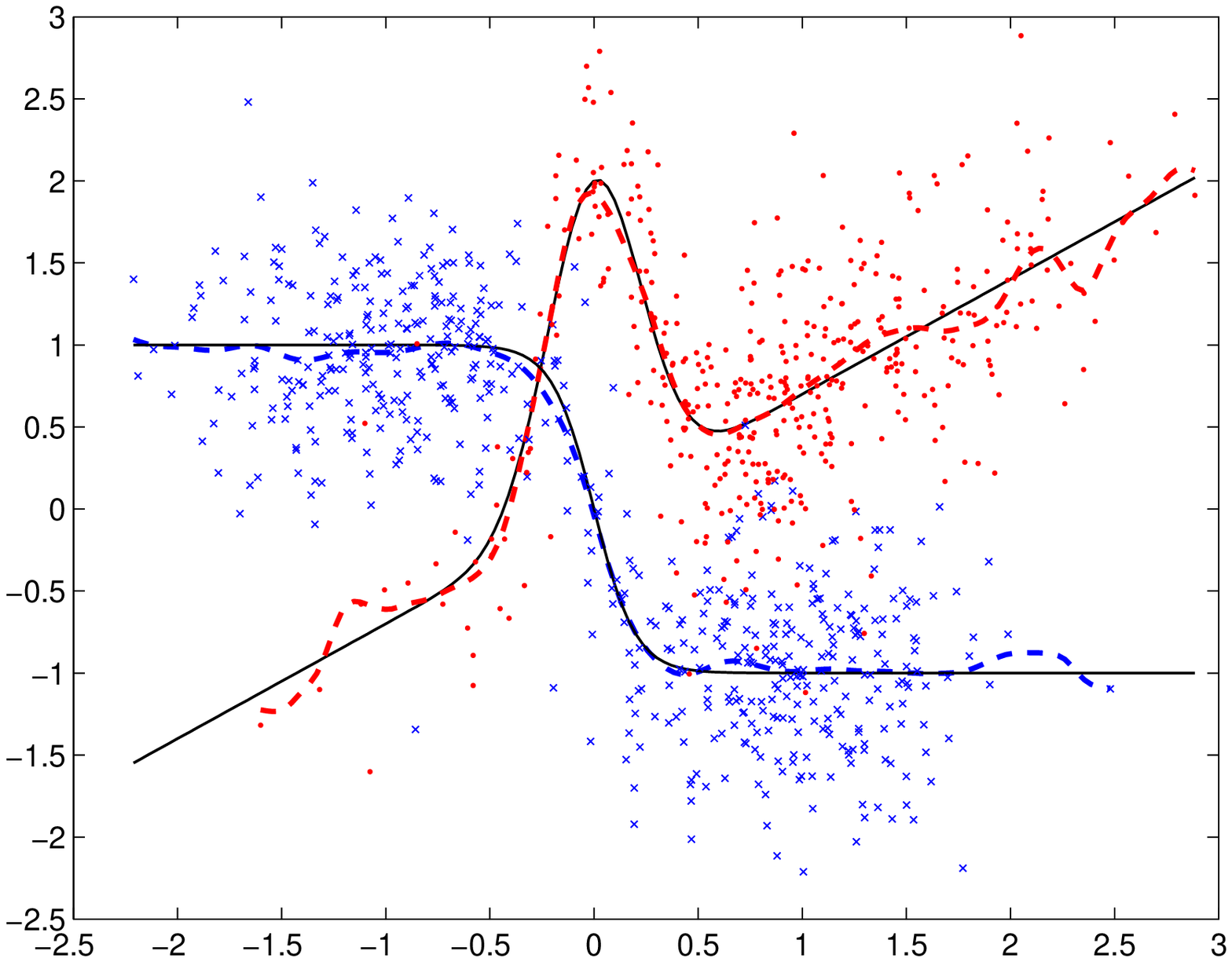}
}
\caption{Parameter estimation for SAEM and scatter plot of simulated data (left). Nonparametric estimation by Robbins-Monro procedure (right).
}\label{fig2}
\end{figure}

Figure \ref{fig2} (left) shows the scatter plot of $Y_k$ against $Y_{k-1}$ and
the linear adjustment. 

We implement our Robbins-Monro procedure with $t= 1:T$ iterations and the smoothing step was defined as 
$$
\gamma_t =\left\{
\begin{array}{ll}
 1& t\leq T_1\\
 (t-T_1)^{-1}&t\geq T_1+1
\end{array}\right. 
.$$
In Figure \ref{fig2} (right) we show the scatter plot of $Y_k$ against $Y_{k-1}$,  $r_1$ and $r_2$ (solid curved lines) and
the respective Robbins-Monro estimates (dotted curved lines) for the ultimate iteration.

\appendix

\section{Proof of technical results}\label{appA}

\begin{proof}[Proof of Lemma \ref{exisdensidad}] 
Take $(e_1,\ldots,e_n)=T(Y_1,\ldots,Y_n)$,
with  $e_k=Y_k-r_{X_k}(Y_{k-1})$, for $k=1,\ldots,n$. Then, the Jacobian matrix of transformation $T$ is triangular and
the absolute value  of the Jacobian is equal to $1$, hence in virtue of the change of variables theorem 
\begin{eqnarray*}
\lefteqn{\mathbb{E}(h(Y_{1:n},X_{1:n},Y_0))}\\
&=&\mathbb{E}(h(T^{-1}(e_{1:n}),X_{1:n},Y_0))\\
&=&\int\sum_{i_{1:n}}h(T^{-1}(u_{1:n}),i_{1:n},y_0))p(e_{1:n}=u_{1:n},X_{1:n}=i_{1:n},Y_0=y_0)du_{1:n}dy_0.
\end{eqnarray*}
From  independence of $Y_0$, $\{X_k\}_{k\geq1}$ and $\{e_k\}_{k\geq1}$, we have the following factorization 
$$
p(e_{1:n}=u_{1:n},X_{1:n}=i_{1:n},Y_0=y_0)=p(e_{1:n}=u_{1:n})p(X_{1:n}=i_{1:n})p(Y_0=y_0)
$$
and by conditions D1 and E6 we obtain
\begin{eqnarray*}
\lefteqn{
\mathbb{E}(h(Y_{1:n},X_{1:n},Y_0))}\\
&=&\!\!\!\int\sum_{i_{1:n}}
h(T^{-1}(u_{1:n}),i_{1:n},y_0)\prod_{k=1}^n\Phi(u_k)
\prod_{k=2}^na_{i_{k-1}i_k}\mu_{i_1}p(Y_0=y_0)du_{1:n}dy_0.
\end{eqnarray*}
Thus, the first result follows.

Integrating the joint density $p(Y_1^n=y_1^n,X_{1:n}=x_{1:n},Y_0=y_0)$ we have

\begin{eqnarray*}
\lefteqn{p(Y_k=y_k,Y_{k'}=y_{k'})=}\\
&&\sum_{x_1^n}\int p(Y_1^n=y_1^n,X_{1:n}=x_{1:n},Y_0=y_0)dy_{0:k-1}dy_{k+1}
\cdots dy_{k'-1}dy_{k'+1:n}. 
\end{eqnarray*}

Since $\Phi(y_k-r_{x_k}(y_{k-1}))\Phi(y_{k'}-r_{x_{k'}}(y_{k'-1}))\leq\|\Phi\|_{\infty}^2$, then using this bound in the above expression 
the remaining integral terms is equal to $1$, thus
$ p(Y_k=y_k, Y_{k'} =y_{k'})\leq \|\Phi\|_{\infty}^2$.
\end{proof}


\begin{proof}[Proof of Proposition \ref{ModeloesMixing}] Let $h_\phi(x_{1:n})=\mathbb{E}(\phi(Y_{0:n})|X_{1:n}=x_{1:n})$. Since we 
assume that $Y_0$, the Markov chain $\{X_k\}_{k\geq1}$ and the sequence $\{e_k\}_{k\geq1}$ are mutually independents, we have by the 
properties of the conditional expectation
$$
\mathbb{E}(\phi(Y_{0:n}))=\mathbb{E}(\mathbb{E}(\phi(Y_{0:n})|X_{1:n}=x_{1:n}))=\mathbb{E}(h_\phi(X_{1:n})).
$$
Since the sequence $\{X_k\}_{k\geq1}$ is strictly stationary under condition E1, then 
$$
\mathbb{E}(\phi(Y_{j:j+n}))=\mathbb{E}(h_\phi(X_{j+1:j+n}))=\mathbb{E}(h_\phi(X_{1:n}))=\mathbb{E}(\phi(Y_{0:n})).
$$
Thus, the sequence $\{Y_k\}_{k\geq0}$ is strictly stationary.

On the other hand, for 
$\phi:\mathbb{R}^k\to[0,1]$ and $\xi:\mathbb{R}^l\to[0,1]$ measurable functions we have
\begin{eqnarray*}
\lefteqn{\mathrm{cov}(\phi(Y_{0:k}),\xi(Y_{k+s:\,k+s+l}))}\\
&=&\mathbb{E}\left( \phi(Y_{0:k})\xi(Y_{k+s:\,k+s+l})\right)-
\mathbb{E}\left( \phi(Y_{0:k})\right)\mathbb{E}\left(\xi(Y_{k+s:\,k+s+l})\right)\\
&=& \mathbb{E}\left( h_\phi(X_{1:k})h_\xi(X_{k+s+1:\,k+s+l})\right)-\mathbb{E}\left( h_\phi(X_{1:k})\right)\mathbb{E}\left(h_\xi(X_{k+s+1:\, k+s+l})\right)\\
&=& \mathrm{cov}(h_\phi(X_{1:k}),h_\xi(X_{k+s+1:\,k+s+l})).
\end{eqnarray*}
Then,
\begin{eqnarray*}
\lefteqn{\alpha_n(Y)}\\
&=&\sup\left\{|\mathrm{cov}(\phi(Y_{0:k}),\xi(Y_{k+s:\,k+s+l}))|:\ 0\leq \phi, \xi \leq 1,\right.\\ &&\hspace{6.4cm} \left. \phi \in \mathcal{M}_{0}^k(Y), \xi\in \mathcal{M}_{k+s}^{k+s+l}(Y)\right\}\\
&=&\sup\left\{|\mathrm{cov}(h_\phi(X_{1:k}),h_\xi(X_{k+s+1:\,k+s+l}))|:\ 0\leq h_{\phi}, h_{\xi}\leq1,\right.\\
&&\hspace{5.8cm} \left. h_{\phi}\in 
\mathcal{M}_{1}^k(X), h_{\xi}\in \mathcal{M}_{k+s+1}^{k+s+l}(X)\right\}\\
&\leq&\alpha_n(X).
\end{eqnarray*}

Now, we only need to prove that $\alpha_n(X)$ decrease geometrically to 0, when $n$ goes to $\infty$. For this, we use that if the initial 
distribution of the strictly stationary Markov chain $X$ is the invariant measure $\mu$ then 
$
\beta_n(X)=\sum_{j=1}^m \sup_{i=1:m}|A_{i,j}^{(n)}-\mu_{i}|\mu_{j},
$
and $\{X_n\}_{n\geq1}$ is geometric ergodic, i.e. there exists $0<\zeta<1$ such that $\sup_i|A_{i,j}^{(n)}-\mu_{i}|\leq m \zeta^n$, obtaining that 
$$
2\alpha_n(Y)\leq 2\alpha_n(X)\leq \beta_n(X)\leq m \zeta^n.
$$
Thus, the result follows.
\end{proof}

\begin{proof}[Proof of Lemma \ref{asympt_cov}]
Considering the variance term, we have
\begin{eqnarray*}
\mathrm{var}(T_{0,n})&=& a^2\,\mathrm{var}\left(K_h\left({y-Y_0}\right)\car_i(X_{1})\right)\\
&&+2ab\,\mathrm{cov} \left(Y_{1}\car_{\{|Y_1|\leq M_n\}}K_h\left({y-Y_0}\right),K_h\left({y-Y_0}\right)\car_i(X_{1})\right)\\
&&+b^2\,\mathrm{var}\left(Y_{1}\car_{\{|Y_1|\leq M_n\}}K_h\left({y-Y_0}\right)\car_i(X_{1})\right). 
\end{eqnarray*}

Hence, 
\begin{eqnarray*}
\lefteqn{\mathrm{var}(T_{0,n})}\\ 
&=&a^2h\!\!\int f_i(y-zh)K^2(z)dz\\
& & +2abh\!\!\int\!\! \mathbb{E}(Y_{k+1}\car_{\{|Y_{k+1}|\leq M_n\}}|X_{k+1}=i, Y_{k}=y-zh ) K^2(z)g_i(y-zh) dz\\
& & +b^2h\!\!\int\mathbb{E}(Y^2_{k+1}\car_{\{|Y_{k+1}|\leq M_n\}}|X_{k+1}=i, Y_{k}=y-zh ) K^2(z) g_{2,i}(y-zh) dz\\ 
& & - h^2\left[a \int f_i(y-zh)K(z)dz\right. \\
&&+ b \left.\int\mathbb{E}(Y_{k+1}\car_{\{|Y_{k+1}|\leq M_n\}}|X_{k+1}=i, Y_{k}=y-zh ) K(z) f_i(y-zh) dz \right]^2.\\
\end{eqnarray*}
From Dominated Convergence Theorem and conditions R2 and R3 we have when $n\to \infty$
$$
\mathrm{Var}(T_{1,n}) \approx h (a^2f_i(y)+2abg_i(y)+b^2g_{2,i}(y))\int K^2(z)dz + o(h^2).
$$

Now, for the covariance terms we define
 $U_k^s=(Y_k\car_{\{|Y_k|\leq M_n\}})^s$, $s=0,1,2$.
As the process is stationary it suffices to consider 
$$
 \mathrm{cov}\left(U_1^sK_h\left({y-Y_0}\right)\car_i(X_1),U_{k+1}^lK_h\left({y-Y_{k}}\right)\car_i(X_{k+1})\right).
$$

Due to $\alpha$-dependency using the inequality of Rio, \cite{Rio}, the covariance is
bounded by
\begin{eqnarray*}
\lefteqn{\mathrm{cov}\left(U_1^sK_h\left({y-Y_0}\right)\car_i(X_1),U_{k+1}^lK_h\left({y-Y_{k}}\right)\car_i(X_{k+1})\right)}\\
&&\hspace{2cm}\leq M_n^{s+l} \mathrm{cov}\left(K_h\left({y-Y_0}\right)\car_i(X_1),K_h\left({y-Y_{k}}\right)\car_i(X_{k+1})\right) \\
&&\hspace{2cm}\leq M_n^{s+l} 4 \|K\|_{\infty}^2 \alpha_{k-1}.
\end{eqnarray*}

This gives
$$
\mathrm{cov}(T_{1,n},T_{k+1,n})\leq (a^2 + 2ab M_n+b^2M_n^2)4\|K\|_{\infty}^2 \alpha_{k-1}.
$$
Set $C_k(i,v)=\{X_{k+1}=i, Y_k=v\}$ and denote by $A^{(k)}_{ij}$ the $(i,j)$-th entry
of the $k$-th power of the matrix $A$. Then,
\begin{eqnarray*}
\lefteqn{\mathbb{E}\left(U_1^sU_{k+1}^lK_h\left({y-Y_0}\right)\car_i(X_1)K_h\left({y-Y_{k}}\right)\car_i(X_{k})\right)}\\
&=&{\int\int
K_h\left({y-u}\right)K_h\left({y-v}\right)
\mathbb{E}\left(U_1^sU_{k+1}^l|C_0(i,u), C_k(i,v) \right)}\\
&&\hspace{4cm}\times  p(Y_0=u,Y_k=v) \mathbb{P}(X_1=i,X_{k+1}=i)dudv\\
&\leq& A^{(k)}_{ii}\mu_i h^2{\int\int
K\left(u\right)K\left(v\right)
\mathbb{E}\left(|U_1^sU_{k+1}^l||C_0(i,y-uh), C_k(i,y-vh)\right)}\\
&&\hspace{5.6cm}\times  p(Y_0=y-uh,Y_k=y-vh) dudv\\
&\leq& A^{(k)}_{ii}\mu_i h^2{\int\int
K\left(u\right)K\left(v\right)
\mathbb{E}\left(|Y_1^sY_{k+1}^l||C_0(i,y-uh), C_k(i,y-vh)\right)}\\
&&\hspace{5.6cm}\times  p(Y_0=y-uh,Y_k=y-vh) dudv.
\end{eqnarray*}
We evaluate in each case  $0\leq s+l\leq 2$, $s,l\in \mathbb{N}$.  
\begin{itemize}
 \item For $s+l=0$,
$$
\mathbb{E}\left(|Y_1^sY_{k+1}^l||X_1=i, Y_0=u, X_{k+1}=i, Y_k=v\right)=1
$$
and in this case
\end{itemize}
$$
\mathrm{cov}\left(U_1^sK_h\left({y-Y_0}\right)\car_i(X_1),U_{k+1}^lK_h\left({y-Y_{k}}\right)\car_i(X_{k+1})\right)
\leq h^2 A^{(k)}_{ii}\mu_i\|\Phi\|_{\infty}^2.
 $$
\begin{itemize}
\item For $s+l=1$  we consider only the case $s=0,\, l=1$, where it holds,
\begin{eqnarray*}
\lefteqn{\mathbb{E}\left(|Y_1^sY_{k+1}^l||X_1=i, Y_0=u, X_{k+1}=i, Y_k=v\right)}\\
&&\hspace{5cm}=\mathbb{E}\left(|Y_{k+1}|| X_{k+1}=i, Y_k=v\right)\\
&&\hspace{5cm}\leq|r_{i}(v)|+\mathbb{E}(|e_{1}|)
\end{eqnarray*}
so that, by continuity of the function $r_{i}(v)$ and the moment condition of $e_1$,
\begin{eqnarray*}
\lefteqn{\mathrm{cov}\left(U_1^sK_h(y-Y_0)\car_i(X_1),U_{k+1}^lK_h\left({y-Y_{k}}\right)\car_i(X_{k+1})\right)}\\
&&\hspace{1.4cm}\preceq h^2 \left((|r_{i}(y)|A^{(k)}_{ii}\mu_i\|\Phi\|_{\infty}^2+ o(h))+\mathbb{E}(|e_{1}|)A^{(k)}_{ii}\mu_i\|\Phi\|_{\infty}^2\right).
\end{eqnarray*}
 
\item For $s+l=2$, assume $s=1,\, l=1$,   
$$
\mathbb{E}\left(|Y_1^sY_{k+1}^l||X_1=i, Y_0=u, X_{k+1}=i, Y_k=v\right)=r_{i,k}(u,v).
$$
Then by continuity of the function $r_{i,k}(u,v)$, then
\begin{eqnarray*}
\lefteqn{\mathrm{cov}\left(U_1^sK_h(y-Y_0)\car_i(X_1),U_{k+1}^lK_h(y-Y_k)\car_i(X_{k+1})\right)}\\
&&\hspace{5.2cm}\preceq h^2r_{i,k}(y,y)A^{(k)}_{ii}\mu_i\|\Phi\|_{\infty}^2+ o(h^3), 
\end{eqnarray*}
and for  $s=0,\, l=2$ the result is followed in the same way exchanging $r_{i,k}(y,y)$ by $r_{i,1}(y,y)$.  
\end{itemize}
We consider the truncated variable $\tilde{\Delta}_k=\Delta_k\car_{\{|Y_{k+1}|\leq M_n\}}$.  It remains to consider the covariance term, 
\begin{eqnarray*}
\mathrm{cov}(T_{1,n},T_{k+1,n}) &=&a^2\,\mathrm{cov}\left(K_h(y-Y_0)\car_i(X_{1}),K_h(y-Y_k)
\car_i(X_{k+1})\right)\\
&&+ab\,\mathrm{cov}\left(K_h(y-Y_0)\car_i(X_{1}),\tilde{\Delta}_k\right)\\
&&+ab\,\mathrm{cov}\left(\tilde{\Delta}_0,K_h(y-Y_k)
\car_i(X_{k+1})\right)\\
&&+b^2\,\mathrm{cov}\left(\tilde{\Delta}_0,\tilde{\Delta}_k\right).
\end{eqnarray*}
By collecting the bounds we obtain,
$$
\mathrm{cov}(T_{1,n},T_{k+1,n})\preceq a^2c_1h^2+2abc_2(y)h^2+b^2c_3(y)h^2.
$$
\end{proof}


\begin{proof}[Proof of Lemma \ref{cvarianza}] Define the truncated kernel estimator of $g_i$,
$$
\tilde{g}_i(y)=\frac{1}{nh}\sum_{k=0}^{n-1}\tilde{\Delta}_k.
$$
Thus,
\begin{eqnarray*}
\lefteqn{\mathbb{P}(|\hat{g}_i(y)-\mathbb{E}\hat{g}_i(y)|\geq 2\epsilon)}\\ 
&&\quad \leq \mathbb{P}(|\tilde{g}_i(y)-\mathbb{E}\tilde{g}_i(y)|\geq \epsilon)+
\mathbb{P}(|\hat{g}_i(y)-\tilde{g}_i(y)-\mathbb{E}(\hat{g}_i(y)-\tilde{g}_i(y)|)\geq\epsilon).
\end{eqnarray*}
By Chebyshev's inequality,
$$
\mathbb{P}(|\hat{g}_i(y)-\tilde{g}_i(y)-\mathbb{E}(\hat{g}_i(y)-\tilde{g}_i(y)|)\geq\epsilon)\leq 
\epsilon^{-2}\mathrm{var}(|\hat{g}_i(y)-\tilde{g}_i(y)|)
$$
and by  definition of the variance, 
$$
\mathrm{var}(|\hat{g}_i(y)-\tilde{g}_i(y)|)\leq \mathbb{E}\left(\hat{g}_i(y)-\tilde{g}_i(y)\right)^2.
$$

We give a bound in the right-hand side of the above inequality using H\"older inequality and the stationary of the model,
\begin{eqnarray*}
\mathbb{E}\left(\hat{g}_i(y)-\tilde{g}_i(y)\right)^2 &\leq & \frac{1}{h^2}
\mathbb{E}\left(Y_1^2 K^2\left(\frac{y-Y_0}{h}\right) \car_{i}(X_{1})\car_{\{|Y_1|> M_n\}}\right)\\
&\leq & h^{-2}\mathbb{E}(|Y_1|^s)\left\|Y_1^{2-s}\car_{\{|Y_1|>M_n\}}K^2 \right\|_{\infty} \\
&\leq& \| K^2 \|_{\infty}M_n^{(2-s)}h^{-2},
\end{eqnarray*}
and 
$$
s_n^2=n^2h^2\mathrm{var}(\tilde{g}_i(y))=n\mathrm{var}(\tilde{\Delta}_0)
+2\sum_{k=1}^{n-1}(n-k)\mathrm{cov}(\tilde{\Delta}_0,\tilde{\Delta}_k).
$$

Now, we will bound the term $s_n^{2}$. First, using point i) of Lemma \ref{asympt_cov} for $a=0$ and $b=1$,
$$
\mathrm{var}(\tilde{\Delta}_0)\preceq  h^2b^2c_3(y)+o(h^3).
$$

Secondly, we use Tran's device to split the covariance 
of $\tilde{\Delta}_k$ into terms:
$$
\sum_{k=1}^{n-1}(n-k)\mathrm{cov}(\tilde{\Delta}_0,\tilde{\Delta}_k)
=\sum_{k=1}^{u_n}(n-k)\mathrm{cov}(\tilde{\Delta}_0,\tilde{\Delta}_k)+\sum_{k=1+{u_n}}^{n-1}(n-k)\mathrm{cov}(\tilde{\Delta}_0,\tilde{\Delta}_k).
$$
In a similar way that in the first bound, applying part ii) of Lemma \ref{asympt_cov} with $a=0$, $b=1$, we obtain for $k\leq u_n$ 
\begin{equation}\label{boundedCov1}
\mathrm{cov}(\tilde{\Delta}_0,\tilde{\Delta}_k)\preceq h^2r_{i,k}(y)\|\Phi\|_{\infty}^2+o(h^3).
\end{equation}

For $k>u_n$ we apply part iii) of Lemma \ref{asympt_cov} obtaining
\begin{equation}\label{boundedCov2}
\mathrm{cov}(\tilde{\Delta}_0,\tilde{\Delta}_k)\leq M_n^2\alpha_{k-1}.
\end{equation}

From inequalities \eqref{boundedCov1}, \eqref{boundedCov2} and taking $u_n=1/h \log n$ we get
\begin{eqnarray*}
\sum_{k=1}^{n-1}(n-k)\mathrm{cov}(\tilde{\Delta}_0,\tilde{\Delta}_k)
&\preceq& c_1 h^2nu_n +n M_n^2\sum_{k\geq u_n}\alpha_{k-1}\\
&\preceq& c_1 h^2nu_n+ c_2 n M_n^2\sum_{k\geq u_n}\zeta^{k-1}\\
&=&c_1h^2nu_n+c_2nM_n^2\frac{\zeta^{u_n}}{1-\zeta}\\
&=&o(nh).
\end{eqnarray*}

Therefore, $s_n^2=O(nh)$.\\

The Fuk-Nagaev's inequality (see Rio \cite{Rio1} theorem 6.2),
applied to random variables $\frac{1}{nh}\sum_{k=1}^n\tilde{\Delta}_k$, 
allows as to obtain
\begin{eqnarray*}
\mathbb{P}(|\tilde{g}_i(y)-\mathbb{E}\tilde{g}_i(y)|\geq \epsilon)
&=& \mathbb{P}\left(\left|\sum_{k=0}^{n-1}\! \tilde{\Delta}_k\right|> \epsilon nh\!\right)\\
&\leq& 4 \left(1+\frac{\lambda^2}{\delta s_n^2}\right)^{-\delta/2}+4nM_n\alpha_{u_n}\lambda^{-1}.
\end{eqnarray*}
 In order to the rate of convergence of the precendent term, we take $4\lambda=\epsilon nh$ obtaining the asymptotic inequality
\begin{eqnarray*}
\mathbb{P}(|\tilde{g}_i(y)-\mathbb{E}\tilde{g}_i(y)|\geq \epsilon)
&\preceq&4\left(1+\frac{\epsilon^2nh}{16 \delta}\right)^{-\delta/2}+\frac{16c M_n\zeta^{u_n}}{\epsilon h}.\\
\end{eqnarray*}

Thus, the result (i) follows. We can prove (ii) in a similar way.
\end{proof}

 
\begin{proof}[Proof of Lemma \ref{lsesgo}]
Set $\Delta_k = Y_{k+1}K_h(y-Y_k)\car_{i}(X_{k+1})$, for $k=0\ldots n-1$. Taking conditional expectation
of $\Delta_k$ given $X_{k+1}=j,Y_{k}=u$,  considering the expression for $r_i(y)$ given in \eqref{regfunction} and using the stationarity
of model, we get
\begin{eqnarray}
\label{Sesgo}
\qquad \mathbb{E}(\Delta_k) 
&=&\mathbb{E}
\left(\mathbb{E}\left(Y_{k+1}K_h(y-Y_k)
\car_{i}(X_{k+1})\left|\right.X_{k+1}=j,Y_{k}=u\right)\right)\\
\nonumber
&=&\int r_i(u)K_h\left({y-u}\right) \mu_i p(Y_0=u)du\\
\nonumber
&=&\int K_h\left({y-u}\right)g_i(u)du.
\end{eqnarray}

Since $\hat{g}_i(y)= \frac{1}{nh}\sum_k \Delta_k$, then the equation (\ref{Sesgo}) implies that
\begin{equation}\label{expect_g}
\mathbb{E}\hat{g}_i(y)=\int K(u)g_i(y-uh)du.
\end{equation}
By second order Taylor's expansion of $g_i$ at $y$ we obtain, 
$$
g_i(y-uh)=g_i(y)-uhg_i'(y)+\frac{(uh)^2}{2}g_i{''}(\tilde{y}_u)
$$
with $\tilde{y}_u=(y-uh)(1-t)+ty$ for some $t\in[0,1]$.\\

As the kernel $K$ is assumed to be of order 2,  substituting the Taylor's approximation into \eqref{expect_g}
gives
$$
\mathbb{E}\hat{g}_i(y)-g_i(y)=\frac{h^2}{2}\int g_i{''}(\tilde{y}_u)u^2 K(u)du.
$$
From condition R2, we have that $g_i{''}$ is continuous. Then  $g_i{''}(\tilde{y})$ converge uniformly  to  $g_i{''}(y)$ 
over the compact set $\mathrm{C}$.
Hence
\begin{equation}
\label{sesgocociente}
\mathbb{E}\hat{g}_i(y)-g_i(y)=\frac{h^2}{2} g_i{''}(y)\int u^2 K(u)du + o(h^{2}).
\end{equation}
Thus,\\
\begin{equation}
\label{sesgocociente1}
\sup_{y\in \mathrm{C}}|\mathbb{E}\hat{g}_i(y)-g_i(y)| =O(h^{2}).
\end{equation}  
The same proof works for the bias of $f_i$, starting from  $$\mathbb{E}\hat{f}_i(y)=\int K(z)p_0(y-zh)\mu_i du.$$ 
\end{proof}


\begin{proof}[Proof of Theorem \ref{conuniforme}]
We start with the following triangle inequality on the positivity set of $f_i(y)$,
\begin{equation}
\label{colomb0}
|\hat{r}_i(y)-r_i(y)|\\
\leq |\hat{g}_i(y)-g_i(y)|\frac{1}{|\hat{f}_i(y)|}+
|\hat{f}_i(y)-f_i(y)|\left|\frac{r_i(y)}{\hat{f}_i(y)}\right|,
\end{equation} 
this implies the following inequality
\begin{eqnarray}
\label{colomb}
\sup_{y\in\mathrm{C}}|\hat{r}_i(y)-r_i(y)|&\leq& \sup_{y\in\mathrm{C}}|\hat{g}_i(y)-g_i(y)|\frac{1}{\inf_{y\in\mathrm{C}}|\hat{f}_i(y)|}\\
\nonumber  && \quad  +
\, \sup_{y\in\mathrm{C}}|\hat{f}_i(y)-f_i(y)|\frac{\sup_{y\in\mathrm{C}}|r_i(y)|}{\inf_{y\in\mathrm{C}}|\hat{f}_i(y)|}.
\end{eqnarray} 
According to the bias-variance decomposition, the proof of the theorem is achieved through the lemmas \ref{cvarianza} and \ref{lsesgo}, 
warranting the existence of strictly positivity of $\inf_{y\in\mathrm{C}}|\hat{f}_i(y)|$.

 
Thus, applying Lemma \ref{cvarianza} with $\epsilon=\epsilon_0\sqrt{\frac{\log n}{nh}}$, $M_n=n^{\gamma}$, $u_n=(h\log n)^{-1}$, and $\delta$ large enough so that $log(n)=o(r)$, we have
\begin{eqnarray*}
\lefteqn{\mathbb{P}\left( |\hat{g}_i(y)-\mathbb{E}\hat{g}_i(y)| \geq \epsilon_0\sqrt{\frac{\log n} {nh}}\right)}\\
&&\hspace{2cm}\preceq  4 \exp\left(-\frac{\epsilon_0^2\log(n)}{32}\right)
+c_1\frac{16n^{\gamma+\frac{1}{2}} \zeta^{u_n}}{\epsilon_0 (\log n)^{\frac{1}{2}}h^{\frac{1}{2}}} +  c_2 \frac{n^{1+\gamma(2-s)}}{\epsilon_0^2 \log(n) h}\\
&&\hspace{2cm}\preceq  n^{-\frac{\epsilon_0^2}{32}}+c_1\frac{16n^{\gamma+ \frac{1}{2}} \zeta^{u_n}}{\epsilon_0(\log n)^{\frac{1}{2}}h^{\frac{1}{2}} } +  c_2 \frac{n^{1+\gamma(2-s)}}{\epsilon_0^2 \log(n) h}.
\end{eqnarray*}
For  $\vartheta=(s-2)\gamma-d-2>0$ and $\frac{\epsilon_0^2}{32} = (s-2)\gamma-d-1$
then,
\begin{equation}\label{ineq_prob}
\mathbb{P}\left(|\hat{g}_i(y)-\mathbb{E}\hat{g}_i(y)|\geq \epsilon_0\sqrt{\frac{\log n}{nh}}\right)\preceq c n^{-(1+\vartheta)}.
\end{equation}
Applying the Borel-Cantelli Lemma, the almost surely pointwise convergence of $|\hat{g}_i(y)-\mathbb{E}\hat{g}_i(y)|$ to $0$ is proved.
We proceed analogously in order to obtain the  almost surely pointwise convergence of  
$|\hat{f}_i(y)-\mathbb{E}\hat{f}_i(y)|\to 0$.

According to Lemma \ref{lsesgo} we have
\begin{eqnarray*}
\inf_{y\in\mathrm{C}}|\hat{f}_i(y)|&\geq & \inf_{y\in\mathrm{C}}|f_i(y)| - \sup_{y\in\mathrm{C}}|\hat{f}_i(y)-\mathbb{E}\hat{f}_i(y)|
-\sup_{y\in\mathrm{C}}|\mathbb{E}\hat{f}_i(y)-f(y)|\\
&\geq&\frac{1}{2}\inf_{y\in\mathrm{C}}|f_i(y)| >0.
\end{eqnarray*}

Thus, the previous results obtained from lemmas \ref{cvarianza} and \ref{lsesgo} and inequality \eqref{colomb0}  given the
pointwise convergence of $|\hat{r}_i(y)-r_i(y)|$.

In order to obtain the uniform convergence on a compact set $\mathrm{C}$, we only need to prove an asymptotic
inequality of type \eqref{ineq_prob} for the term $\sup_{y\in\mathrm{C}}|\hat{g}_i(y)-\mathbb{E}\hat{g}_i(y)|$, and 
analogously for $\sup_{y\in\mathrm{C}}|\hat{f}_i(y)-\mathbb{E}\hat{f}_i(y)|$, in inequality \eqref{colomb}. For this,
we proceed by using truncation device as in Ango Nze {\it et. al.} \cite{Ango-Doukhan-Buhulmann}, assuming the moment condition M1.

Let us set $\Delta_k=Y_{k+1}K_h(y-Y_k)\car_{i}(X_{k+1})$ and the truncated variable
$\tilde{\Delta}_k=\Delta_k\car_{\{|Y_{k+1}|\leq M_n\}}$. Then, we define the truncated kernel estimator of $g_i$ by
$$
\tilde{g}_i(y)=\frac{1}{nh}\sum_{k=0}^{n-1}\tilde{\Delta}_k.
$$
Since $\|K\|_{\infty}<\infty$, taking $M_n=M_0\log n$, then clearly we obtain
$$
\mathbb{P}\left( \sup_{y\in\mathrm{C}}| \hat{g}_i(y)-\tilde{g}_i(y)| >0\right)\leq n\mathbb{P}(|Y_1|> M_0\log n )
\leq \mathbb{E}(\exp(|Y_1|)) n^{1-M_0}
$$
and, by the Cauchy-Schwarz inequality and condition R3, 
\begin{eqnarray*}
\lefteqn{\sup_{y\in\mathrm{C}} \mathbb{E}\left(|\hat{g}_i(y)-\tilde{g}_i(y)|\right)}\\
&&\hspace{1cm}\leq
\frac{1}{h}\sup_{y\in\mathrm{C}} \mathbb{E}\left(|Y_1|\car_{\{|Y_1|>M_0\log(n)\}}K_h\left({y-Y_0}\right)\car_i(X_1)\right)\\
&&\hspace{1cm} \leq\frac{1}{h} n^{-M_0/2}  \mathbb{E}(\exp(|Y_1|))^{1/2} \sup_{y\in\mathrm{C}}\mathbb{E}\left(|Y_1|^2K_h^2\left({y-Y_0}\right)\car_i(X_1)\right)^{1/2}\\
&&\hspace{1cm}\leq c_4  \frac{n^{-M_0/2}}{h^{1/2}}.
\end{eqnarray*}
Now, we reduce computations to a chaining
argument, (see \cite{Ferraty}, p\'ags. 32 and 78) 
for the case of a kernel estimator with bounded variables. Let  $\mathrm{C}$ be covered by a 
finite number $\nu_n$ of intervals $B_k$ with diameter $2L_n$ and center at $t_k$. Then,
\begin{eqnarray*}
\lefteqn{\sup_{y\in\mathrm{C}}|\tilde{g}_i(y)-\mathbb{E}\tilde{g}_i(y)|}\\
&\leq&\max_{k=1,\ldots,\nu_n}|\tilde{g}_i(t_k)-\mathbb{E}\tilde{g}_i(t_k)|
+\sup_{y\in\mathrm{C}}|\tilde{g}_i(t_k)-\tilde{g}_i(y)|+\sup_{y\in\mathrm{C}}|\mathbb{E}\tilde{g}_i(y)-\mathbb{E}\tilde{g}_i(t_k)|
 \end{eqnarray*}

Now, let us examine each term in the right-hand side above. First, we have from Lemma \ref{cvarianza}
\begin{eqnarray*}
\lefteqn{\mathbb{P}\left(\max_{k=1,\ldots,\nu_n}|\tilde{g}_i(t_k)-
\mathbb{E}\tilde{g}_i(t_k)|>\frac{\varepsilon_0}{2}\sqrt{\frac{\log n}{nh}}\right)}\\
&&\hspace{2.5cm} \leq \sum_{k=1}^{\nu_n}\mathbb{P}\left(|\tilde{g}_i(t_k)-\mathbb{E}\tilde{g}_i(t_k)|>\frac{\varepsilon_0}{2}\sqrt{\frac{\log n}{nh}}\right)\\
&&\hspace{2.5cm} \preceq  \nu_n \left(4n^{-\frac{\epsilon_0^2}{128}}+c_1\frac{ 32 n^{1/2}M_n \zeta^{u_n}}{\epsilon_0(\log n)^{1/2}h^{1/2}}\right).
\end{eqnarray*}
For the second and third terms, we use the following inequality obtained from condition R1, 
\begin{eqnarray*}
|\tilde{g}_i(t_k)-\tilde{g}_i(y)|&\leq& M_n\frac{1}{nh}
\sum_{k=1}^n\left| K_h\left({t_k-Y_k}\right) -K_h(y-Y_k) \right|\\
&\leq& c \frac{M_n}{h^{1+\beta}}|y-t_k|^\beta\leq c \frac{M_nL_n^{\beta}}{h^{1+\beta}},
\end{eqnarray*}
for some constants $c, \beta >0$.
Therefore, 
\begin{eqnarray*}
\lefteqn{\mathbb{P}\left(\sup_{y\in\mathrm{C}}|\hat{g}_i(y)-\mathbb{E}\hat{g}_i(y)|\geq \epsilon_0\sqrt{\frac{\log n}{nh}}\right)}\\
& \preceq & \nu_n \left(4n^{-\frac{\epsilon_0^2}{128}}+c_1\frac{ 32 n^{1/2}M_n \zeta^{u_n}}{\epsilon_0(\log n)^{1/2}h^{1/2}}\right) +
2\mathbb{P}\left(\frac{M_nL_n^{\beta}}{h^{1+\beta}}\geq\frac{\epsilon_0}{2}\sqrt{\frac{\log n}{nh}}\right)\\
&&+\, c_4 \frac{n^{-M_0/2}}{h^{1/2}}.
\end{eqnarray*}
Let us set $L_n^{\beta}=n^{-\frac{1}{2}}h^{\frac{1}{2}+\beta}M_n^{-1}$ and $\nu_n=c_5/L_n$ 
we obtain 
\begin{eqnarray*}
\lefteqn{\mathbb{P}\left(\sup_{y\in\mathrm{C}}|\hat{g}_i(y)-\mathbb{E}\hat{g}_i(y)|\geq \epsilon_0\sqrt{\frac{\log n}{nh}}\right)}\\
&&\hspace{1.5cm} \preceq \frac{c_5}{L_n}  \left(4n^{-\frac{\epsilon_0^2}{128}}+c_1\frac{ 32 n^{1/2}M_n \zeta^{u_n}}{\epsilon_0(\log n)^{1/2}h^{1/2}}\right)+  c_4 \frac{n^{-M_0/2}}{h^{1/2}}.
\end{eqnarray*}

Taking $u_n=(h\log n)^{-1}$, $\vartheta=\frac{\epsilon_0^2}{128} +\frac{1+d}{2\beta} + d-1>0$, and $M_0=2(\vartheta+1)+d$. We have
\begin{equation}\label{ineq_prob_1}
\mathbb{P}\left(\sup_{y\in\mathrm{C}}|\hat{g}_i(y)-\mathbb{E}\hat{g}_i(y)|\geq \epsilon_0\sqrt{\frac{\log n}{nh}}\right)\preceq
c n^{-(1+\vartheta)}.
\end{equation}
Hence, the Borel-Cantelli Lemma implies the almost surely convergence of term $\sup_{y\in\mathrm{C}}|\hat{g}_i(y)-\mathbb{E}\hat{g}_i(y)|$. 

The uniform convergence over a compact set of the regression function  $\hat{r}_i$  follows in the same way that for the a.s. 
pointwise convergence.
\end{proof}

\begin{remark} Note that in the proof of the a.s. pointwise convergence, the probability term in \eqref{ineq_prob} is sumable if 
$\vartheta=(s-2)\gamma-d-2>0$. This is only possible if $s>2$, and so the restriction imposed in condition M2 arises.  

From the asymptotic inequalities \eqref{ineq_prob} and \eqref{ineq_prob_1} we can notice that the convergence rate in Theorem~\ref{conuniforme}
is $\sqrt{\frac{\log n}{ nh }}$.\\
\end{remark}


\begin{proof}[Proof of Lemma \ref{lema_GradEst}]
 Taking expectation in (\ref{potencialU}), it follows that (\ref{constrastelimite}) is true. 
For the second part, we use simply the fact that the potential $U$ is absolutly integrable with respect to the measure $\mathbb{P}(X_{1:n}^{t}=x|Y_{0:n},\theta')\mu_{c}(dx)$ with $\mu_{c}(dx)$   the counting measure on $\{1,\ldots, m\}^n$. So, by the dominated convergence theorem we have 

\begin{eqnarray*}
\lefteqn{\mathbb{E}_{\theta'}\left(\nabla_{\theta}U(y,Y_{1:n},X_{1:n}^{t},\theta)|\mathcal{F}_{t-1})\right)}\\
&&\hspace{3.5cm}=\int\nabla_{\theta}U(y,Y_{1:n},x,\theta)\mathbb{P}(X_{1:n}^{t}=x|Y_{0:n},\theta')\mu_{c}(dx) \\
&&\hspace{3.5cm}=\nabla_{\theta} \int U(y,Y_{1:n},x,\theta)\mathbb{P}(X_{1:n}^{t}=x|Y_{0:n},\theta')\mu_{c}(dx)\\
&&\hspace{3.5cm}=\nabla_{\theta}u(y,Y_{1:n},\theta).\\
\end{eqnarray*}
\end{proof}





\begin{thebibliography}{9}

\bibitem{ailliot}
P.~Ailliot and V.~Monbet.
\newblock {Markov-switching autoregressive models for wind time series.}
\newblock {\em Environmental Modelling \& Software}, 30:92–101, 2012.

\bibitem{Ango-Doukhan-Buhulmann}
P.~Ango-Nze, P.~Buhlmann, and P.~Doukhan.
\newblock { Weak dependence beyond mixing and asymptotics for nonparametric
  regression.}
\newblock {\em Annals of Statistics}, 30:397--430, 2002.

\bibitem{Baum70}
L.~E. Baum, T.~Petrie, G.~Soules, and N.~Weiss.
\newblock {A maximization tecnique occuring in the statistical analysis of a
  probabilistic functions of Markov chains}.
\newblock {\em Ann. Math. Stat.}, 41:164--171, 1970.

\bibitem{benaglia}
T.~Benaglia, D.~Chauveua, and DR. Hunter.
\newblock {An EM-like algorithm for semi-and non-parametric estimation in
  multivariate mixtures}.
\newblock {\em Journal of computional and graphical statistics},
  18(2):505--526, 2009.

\bibitem{cappe-moulines-ryden}
O.~Capp\'e, E.~Moulines, and T.~Ryd\'en.
\newblock {\em {\it Inference in Hidden Markov Models }}.
\newblock Springer-Verlag, 2005.

\bibitem{Carterkohn}
C.~K Carter and R.~Kohn.
\newblock {On Gibbs sampling for state space model}.
\newblock {\em Biometrika}, 81:541--553, 1994.

\bibitem{DeCastro}
Y.~De Castro, E.~Gassiat, C.~Lacour.
\newblock {Minimax adaptive estimation of non-parametric hidden markov models}.
\newblock {\em 	arXiv:1501.04787 [math.ST]}

\bibitem{Lavielle}
B.~Delyon, M.~Lavielle, and E.~Moulines.
\newblock {Convergence of a stochastic approximation version of EM algorithm}.
\newblock {\em The Annals of Statistics}, 27(1):94--128, 1999.

\bibitem{douc}
R.~Douc, E.~Moulines, and T.~Ryd{\'e}n.
\newblock {Asymptotic properties of the maximum likelihood estimator in
  autoregressive models with Markov regime}.
\newblock {\em Ann. Statist.}, 32:2254--2304, 2004.

\bibitem{dmixing}
P.~Doukhan.
\newblock {\em {\it Mixing: Propierties and Examples.}}, volume~85.
\newblock Lecture Notes in Statist., 1994.

\bibitem{duflo}
M.~Duflo.
\newblock {\em {\it Algorithmes Stochastiques}}.
\newblock Springer-Verlag, Berlin, 1996.

\bibitem{Ferraty}
F.~Ferraty, N.~Ant\'on, and P.~Vieu.
\newblock {\em {\it Regresi\'on No param\'etrica: Desde la Dimensi\'on Uno
  hasta la dimensi\'on Infinita}}.
\newblock Servicio editorial de Universidad del Pa\'{\i}s Vasco, 2001.

\bibitem{francq-Roussignol}
C.~Francq and M.~Roussignol.
\newblock {On white noises driven by hidden Markov Chains}.
\newblock {\em J. Time Ser. Anal.}, 18:553--578, 1997.

\bibitem{francq-Roussignol-Zakoian}
C.~Francq, M.~Roussignol, and J-M Zakoian.
\newblock {Conditional heteroskedasticity driven by hidden Markov Chains}.
\newblock {\em J. Time Ser. Anal.}, 2:197--220, 2.

\bibitem{franke-stockis}
J.~Franke, J.~P. Stockis, J.~Tadjuidje, and W.K. Li.
\newblock Mixtures of nonparametric autoregressions.
\newblock {\em Journal of Nonparametric Statistics}, 23(2):287--303, 2011.

\bibitem{Goldfeld}
S.~M. Goldfeld and R.~Quandt.
\newblock {A Markov Model for Switching Regressions}.
\newblock {\em Journal of Econometrics}, 1:3--16, 1973.

\bibitem{Hamilton}
J.D. Hamilton.
\newblock {A new approach to the economic analysis of non stationary time
  series and the business cycle}.
\newblock {\em Econometrica}, pages 357--384, 1989.

\bibitem{Hamilton-Raj}
J.D. Hamilton and B.~Raj.
\newblock {\em {\it Advances in Markov-Switching Models: Applications in
  Business Cycle Research and Finance (Studies in Empirical Economics)}}.
\newblock Springer, 2003.

\bibitem{Harel-Puri}
M.~Harel and M.~Puri.
\newblock {U-statistiques conditionnells universellement consistantes pour des
  mod\`eles de Markov cach\'es}.
\newblock {\em S. R. Acad. Sci. Paris, S\'erie I}, 333:953--956, 2001.

\bibitem{Kemeny}
J.~G. Kemeny and J.~L. Snell.
\newblock {\em {\it Finite Markov Chains}}.
\newblock Van Nostrand, Princenton, New Jersey, 1960.

\bibitem{Kim-Nelson}
C.~Kim and C.~Nelson.
\newblock {\em {\it State-Space Models with Regime Switching Classical and
  Gibbs-Sampling Approaches with Applications.}}
\newblock MIT Press, 1999.

\bibitem{Krish}
V.~Krishnamurthy and T.~Rydén. 
\newblock {Consistent estimation of linear and non-linear autoregressive models with Markov regime.} 
\newblock Journal of Time Series Analysis 19 (1998), 291-307.

\bibitem{Krolzig}
H-M. Krolzig.
\newblock {\em {\it Markov-Switching Vector Autoregressions: Modelling,
  Statistical Inference, and Application to Business Cycle Analysis}}.
\newblock Springer, 1997.

\bibitem{Polyakuditsky}
B.~T. Polyak and A.~B. Juditsky.
\newblock {Acceleration of stochastic approximation by averaging}.
\newblock {\em SIAM J. Control Optim.}, 30:838--855, 1992.

\bibitem{Rio}
E.~Rio.
\newblock {Covariance inequalities for strongly mixing processes}.
\newblock {\em Annales de l'institut Henri Poincar\'{e} (B) Probabilit\'es et
  Statistiques}, 29:587--597, 1993.

\bibitem{Rio1}
E.~Rio.
\newblock {\em {\it Th\'eorie asymptotique des processus faiblement
  d\'ependents}}, volume~31.
\newblock Springer-SMAI: Paris., 2000.

\bibitem{luis1}
R.~R\'{\i}os and L.~A. Rodr\'{\i}guez.
\newblock Estimaci\'on semiparam\'etrica en procesos autorregresivos con
  r\'egimen de Markov.
\newblock {\em Divulgaciones Matem\'aticas}, 16(1):155--171, 2008.

\bibitem{Rosales}
R.~Rosales.
\newblock {MCMC for hidden Markov models incorporating aggregation of states
  and filtering}.
\newblock {\em Bulletin of Mathematical Biology}, 66(5):1173--1199, 2004.

\bibitem{YaoRe}
J.~Yao.
\newblock {On Recursive Estimation in Incomplete Data Models}.
\newblock {\em Statistics}, 34:27--51, 2000.

\bibitem{Yao}
J.~Yao and J.~G. Attali.
\newblock {On stability of nonlinear AR process with Markov switching}.
\newblock {\em Adv. Applied Probab}, 32:394--407, 1999.

\end{thebibliography}
\end{document}